\documentclass[a4paper,twoside,11pt,reqno]{amsart}
\newcommand{\Ueberschrift}{Convex Fujita numbers and the Kodaira--Enriques classification of surfaces}
\newcommand{\Kurztitel}{Convex Fujita numbers of minimal surfaces}
\usepackage{amsmath} \usepackage{amssymb} \usepackage{amsopn}
\usepackage{amsthm} \usepackage{amsfonts} \usepackage{mathrsfs}
\usepackage{latexsym}
\usepackage[headings]{fullpage}
\usepackage{mathtools}
\usepackage[all]{xy}
\usepackage[dvips]{graphicx} 
\usepackage[usenames]{color}
\usepackage[T1]{fontenc}
\usepackage[utf8]{inputenc}
\usepackage{mathabx}
\usepackage{stmaryrd}

\usepackage{setspace,fancyhdr}
\usepackage[margin=2.5cm]{geometry}
\usepackage[nobysame,alphabetic]{amsrefs}
\usepackage{framed}
\usepackage{tikz}
\usepackage{enumitem}
\usepackage{booktabs}
\usepackage[pdfpagelabels, pdftex]{hyperref}
\hypersetup{
  pdftitle={},
  pdfauthor={},
  pdfsubject={},
  pdfkeywords={},
  colorlinks=true,    
  linkcolor=blue,     
  citecolor=blue,     
  filecolor=blue,      
  urlcolor=blue,       
  breaklinks=true,
  bookmarksopen=true,
  bookmarksnumbered=true,
  pdfpagemode=UseOutlines,
  plainpages=false}

\DeclareMathOperator{\rE}{E}

\DeclareMathOperator{\rH}{H}

\DeclareMathOperator{\rT}{T}

\DeclareMathOperator{\rh}{h}

\newcommand{\bA}{{\mathbb A}}

\newcommand{\bC}{{\mathbb C}}

\newcommand{\bF}{{\mathbb F}}

\newcommand{\bP}{{\mathbb P}}
\newcommand{\bQ}{{\mathbb Q}}

\newcommand{\bZ}{{\mathbb Z}}

\newcommand{\cE}{{\mathscr E}}
\newcommand{\cF}{{\mathscr F}}

\newcommand{\cH}{{\mathscr H}}
\newcommand{\cI}{{\mathscr I}}

\newcommand{\cL}{{\mathscr L}}
\newcommand{\cM}{{\mathscr M}}

\newcommand{\cO}{{\mathscr O}}

\newcommand{\dO}{{\mathcal O}}

\newcommand{\surj}{\twoheadrightarrow} 
\newcommand{\inj}{\hookrightarrow}

\DeclareMathOperator{\pr}{pr}

\DeclareMathOperator{\Hom}{Hom}

\DeclareMathOperator{\Aut}{Aut}
\DeclareMathOperator{\End}{End}

\DeclareMathOperator{\coker}{coker}

\DeclareMathOperator{\GL}{GL}

\DeclareMathOperator{\SL}{SL}

\DeclareMathOperator{\Sp}{Sp}

\DeclareMathOperator{\Pic}{Pic}
\DeclareMathOperator{\NS}{NS}
\DeclareMathOperator{\Num}{Num}

\newcommand{\Gm}{\bG_m}

\DeclareMathOperator{\RR}{R}
\DeclareMathOperator{\Ext}{Ext}

\DeclareMathOperator{\cExt}{\cE\mathit{xt}}
\DeclareMathOperator{\cHom}{\cH\mathit{om}}

\DeclareMathOperator{\res}{res}

\newcommand{\ph}{\varphi}

\newcommand{\tors}{{\rm tors}}
\newcommand{\ab}{{\rm ab}}

\DeclareMathOperator{\Sym}{Sym}

\newtheorem{thm}{Theorem}[section]

\newtheorem{theorem}[thm]{Theorem}

\newtheorem{prop}[thm]{Proposition}
\newtheorem{lem}[thm]{Lemma}

\newtheorem{cor}[thm]{Corollary}

\newtheorem{thmABC}{Theorem}

\newtheorem{propABC}[thmABC]{Proposition}

\theoremstyle{definition}
\newtheorem{defi}[thm]{Definition}

\theoremstyle{remark}
\newtheorem{rmk}[thm]{Remark}

\newenvironment{pro*}[1][Proof]{{\it{#1:}} }{}
\newenvironment{pro**}[1][]{{\it{#1}} }{\hfill $\square$}

\numberwithin{equation}{section}
\newcommand{\tref}[1]{Theorem~\ref{#1}}
\newcommand{\secref}[1]{\S\ref{#1}}

\newcommand{\cref}[1]{Corollary~\ref{#1}}
\newcommand{\dref}[1]{Definition~\ref{#1}}

\newcommand{\lref}[1]{Lemma~\ref{#1}}
\newcommand{\pref}[1]{Proposition~\ref{#1}}
\newcommand{\rref}[1]{Remark~\ref{#1}}

\def\mc{\mathscr}

\DeclareMathOperator{\Bl}{Bl}

\def\Gm{\mathbb{G}_m}

\def\pr{\text{pr}}

\def\P{\mathbb{P}}
\def\bP{\P}

\def\cE{\mc{E}}
\def\cF{\mc{F}}

\def\cH{\mc{H}}
\def\cI{\mc{I}}

\def\cL{\mc{L}}
\def\cM{\mc{M}}

\def\cO{\mc{O}}

\DeclareMathOperator{\conFN}{Fu} 

\definecolor{intOrange}{rgb}{1.0,.310,.0}
\usepackage[textsize=tiny]{todonotes}

\begin{document}

\hrule width\hsize

\vskip 0.5cm

\title[\Kurztitel]{\Ueberschrift} 

\author{Jiaming Chen}
\address{Jiaming Chen, Institut f\"ur Mathematik, Goethe--Universit\"at Frankfurt, Robert-Mayer-Stra\ss e {6--8},
60325~Frankfurt am Main, Germany} 
\email{\tt chen@math.uni-frankfurt.de}

\author{Alex K\"{u}ronya}
\address{Alex K\"uronya, Institut f\"ur Mathematik, Goethe--Universit\"at Frankfurt, Robert-Mayer-Stra\ss e {6--8},
60325~Frankfurt am Main, Germany} 
\email{\tt kuronya@math.uni-frankfurt.de}

\author{Yusuf Mustopa}
\address{Yusuf Mustopa, University of Massachusetts Boston, Department of Mathematics, Wheatley Hall, 100 William T Morrissey Blvd, Boston, MA 02125, USA}
\email{Yusuf.Mustopa@umb.edu}

\author{Jakob Stix}
\address{Jakob Stix, Institut f\"ur Mathematik, Goethe--Universit\"at Frankfurt, Robert-Mayer-Stra\ss e {6--8},
60325~Frankfurt am Main, Germany} 
\email{\tt stix@math.uni-frankfurt.de} 
	
\thanks{The authors acknowledges support by Deutsche  Forschungsgemeinschaft  (DFG) through the Collaborative Research Centre TRR 326 "Geometry and Arithmetic of Uniformized Structures", project number 444845124.}

\date{October 25, 2023} 

\maketitle

\begin{quotation} 
\noindent \small {\bf Abstract} --- We contemplate  the range of  convex Fujita numbers for minimal smooth projective surfaces according to their position in the Kodaira--Enriques classification. 
\end{quotation}

\DeclareRobustCommand{\SkipTocEntry}[5]{}
\setcounter{tocdepth}{1} {\scriptsize \tableofcontents}

\section{Introduction}
\label{sec:intro}

\subsection{Motivation}

Our main goal is to study effective global generation of adjoint line bundles on minimal surfaces across the Kodaira--Enriques classification. Effective positivity questions  (for instance in the form of conjectures by Fujita, Mukai, and Kawamata)  have been around in birational geometry for several decades now, however, they appear to be beyond our reach in general. Initiated in \cite{chen_convex_2023}, we advocate an approach which aims at more precise information at the expense of a potentially more modest generality at first. 

The center of our investigations is the  concept of the convex Fujita number of a smooth projective variety. Given such a variety $X$, 
we define its convex Fujita number $\conFN(X)$ as  the minimal $m \geq 0$ such that  for all $s \geq m$ and any ample divisors $L_1, \ldots, L_s$ on $X$ the adjoint divisor 
\[
K_X + L_1 + \ldots + L_s
\]
is globally generated.

Convex Fujita numbers are  understood as a measure of effective positivity of line bundles on a smooth projective variety. While Fujita's freeness conjecture predicts that $K_X+mL$ is base-point free for all $m\geq \dim X+1$ and all ample divisors $L$ on $X$, which would roughly translate to $\conFN(X)\leq \dim X+1$, our variant of the problem provides means for a finer classification.

The purpose of this work is to understand the range of the convex Fujita number for minimal smooth projective surfaces in view of the Kodaira--Enriques classification. Reider's method (cf. \secref{sec:reider}), relating base points of numerically large adjoint linear systems to Bogomolov unstable rank $2$ vector bundles shows   for smooth projective surfaces $X$ the inequalities
\[
0 \leq \conFN(X) \leq 3.
\]
Focusing  on minimal surfaces appears to be justified as  we show in Theorem~\ref{prop:cofinalFN3} that any smooth projective surface $X$ admits a birational modification  $X' \to X$ such that $X'$ has convex Fujita number $\conFN(X') = 3$. This supports the idea that blowing up points typically  forces Fujita numbers to become maximal possible. For a  background on  convex Fujita numbers and the circle of ideas around  Fujita's freeness conjecture  we refer the reader  to \cite{chen_convex_2023}.

\subsection{Convex Fujita numbers of minimal surfaces} 

The concrete geometry of algebraic surfaces combined with the powerful method of Reider allows us to determine the convex Fujita number in many cases. 

\begin{thmABC}
\label{thmABC:FNminimalsurfaces}
In Kodaira dimension $\leq 0$ the convex Fujita number of a (relatively) minimal smooth projective surfaces is determined as follows. The convex Fujita number 
\begin{enumerate}[label=(Kodaira dim \arabic*) \ \ ,align=left,labelindent=0pt,leftmargin=*]
\setcounter{enumi}{-1}
\item[(Kodaira dim $<0$)]
\begin{enumerate}[label=(\arabic*),align=left,labelindent=0pt,leftmargin=*]
\item 
of $\bP^2$ is $3$, 
\item 
of a Hirzebruch surface is $2$,
\item
of a ruled surface is $2$ or $3$,
\end{enumerate}
\item
\begin{enumerate}[label=(\arabic*),align=left,labelindent=0pt,leftmargin=*, resume]
\item
of an abelian surface is $0$ or $2$,
\item
of a bielliptic surface is $2$,
\item
of a K3 surface is $0$ or $2$,
\item
of an Enriques surface is $1$ or $2$,
\end{enumerate}
\end{enumerate}
Moreover, there are minimal surfaces
\begin{enumerate}[label=(Kodaira dim \arabic*) \ \ ,align=left,labelindent=0pt,leftmargin=*, resume]
\item
of Kodaira dimension $1$ and convex Fujita number $2$ and $3$, and 
\item
of general type with convex Fujita number $0$, $2$ and $3$.
\end{enumerate}
\end{thmABC}

\begin{rmk}
At the moment we cannot decide whether there are minimal smooth projective surfaces of convex Fujita number $n$ and Kodaira dimension $\kappa$ for $(\kappa,n)$ in the list $(1,0)$, $(1,1)$ and $(2,1)$. 

However, as shown in Proposition~\ref{prop:surface general type FN1}, certain ramified double covers $X$ of principally polarized abelian surfaces with Picard number $1$ yield smooth projective surfaces that are minimal and of general type with $\conFN(X)=1$
\emph{assuming} that the Picard number of $X$ is again $1$. 
\end{rmk}

The proof of \tref{thmABC:FNminimalsurfaces}  occupies sections \S\secref{sec:FN rational ruled}--\ref{sec:FN Kodaira 2}. The individual results are more precise than \tref{thmABC:FNminimalsurfaces} in that we can often  describe the geometry on a particular minimal surface that decides the convex Fujita number in case there are options. We list  these more detailed results below.

\begin{propABC}[see \pref{prop:FNruled surface}]
Let $X=\bP(\cE)$ be the ruled surface associated to a rank $2$ vector bundle $\cE$ on a smooth projective curve. 
The convex Fujita number $\conFN(X)$ equals $3$ if $\cE$ is stable or of odd degree. In all other cases $\conFN(X) = 2$.
\end{propABC}

\begin{propABC}[see \pref{prop:FNabeliansurface}]
Let $X$ be an abelian surface. Then we have
\[
\conFN(X) = 2 \quad \iff \quad \text{ $X$ has an ample $\cL$ with $(\cL^2) \leq 4$}.
\]
If the above assertions do not hold, then we have $\conFN(X) = 0$. 
\end{propABC}

\begin{propABC}[see \pref{prop:FN for K3}]
Let $X$ be a K3 surface. 
Then $\conFN(X) = 0$ unless there exist an elliptic fibration $\ph: X \to \bP^1$ with general fiber $E$ and image   of a section $S$ such that all fibers are irreducible and reduced. In the latter case the line bundle $\cL = \dO_X(mE + S)$, for $m \geq 3$, is ample but not globally generated, so we have $\conFN(X) = 2$.
\end{propABC}

An Enriques surface $X$  is called unnodal, if there is no smooth rational curve on $X$.

\begin{propABC}[see \pref{prop:FN Enriques}, \ref{prop:FN Enriques unnodal} and \ref{prop:FN1 Enriques}]
Let $X$ be an Enriques surface. 
\begin{enumerate}[align=left,labelindent=0pt,leftmargin=*]
\item 
Then $\conFN(X) = 2$ if and only if $X$ admits a genus one fibration $f: X \to \bP^1$ with a bisection  that meets every component of a fibre of $f$.  This applies in particular to unnodal Enriques surfaces.
\item
Otherwise we have $\conFN(X) = 1$, and this occurs for certain $(\tau,\bar \tau)$-generic Enriques surfaces (terminology of \cite{brandhorst_automorphism_2022}) having the property that all genus one fibrations have at least one fibre with at least $3$ irreducible components.
\end{enumerate}
\end{propABC}

Surfaces of positive Kodaira dimension are dealt with in sections \secref{sec:FN Kodaira 1} and \secref{sec:FN Kodaira 2}.

\begin{propABC}[see \pref{prop:elliptic fibration FN2}  and \ref{prop:elliptic fibration FN3}]
On minimal surfaces of Kodaira dimension $1$. Let $C$ a smooth projective curve. 
\begin{enumerate}[align=left,labelindent=0pt,leftmargin=*]
\item 
An elliptic fibration $X \to C$ with a section has $\conFN(X) = 2$ if all fibers are irreducible and reduced and $\chi(X,\dO_X)$ is even.
\item
Let $E$ be an elliptic curve, and let $C \to \bP^1$ be a branched cover with Galois group $G = E[2]$ with $C$ of genus $2$ and such that  $E$ is not an isogeny factor of $\Pic^0(C)$. 
Then $X = E \times C/G$, with $G$ acting by translation on $E$, is an isotrivial elliptic fibration and a minimal smooth projective surface of Kodaira dimension $1$ and convex Fujita number $\conFN(X) = 3$. 
\end{enumerate}
\end{propABC}

\begin{propABC}[see \pref{prop:surface general type FN0 simply connected} 
\pref{prop:products of curves}  and \rref{rmk:QuinticGodeauxFN3}]
On minimal surfaces of  Kodaira dimension $2$.

There exist minimal surfaces $X$ of general type with $\conFN(X)=0,2,3$. More concretely: 
\begin{enumerate}[align=left,labelindent=0pt,leftmargin=*]
\item   \label{propitem:gentypeFN0}
A very general hyperplane $X$ in $\bP^3$ of degree $d \geq 5$ is minimal of general type with convex Fujita number $\conFN(X) = 0$.

\item \label{propitem:gentypeFN2}
The product $X= C_1 \times C_2$ of smooth projective curves of genus at least $2$ is minimal of general type with convex Fujita number $\conFN(X) = 2$.

\item \label{propitem:gentypeFN3}
Let $Y$ be a smooth quintic in $\bP^3$ with coordinates $x=[x_1:x_2:x_3:x_4]$  cut out by 
\[
F(x) = \sum_{i=1}^4 a_i x_i^5 + \sum_{i=1}^4 b_i x_i^3 x_{3i} x_{9i} + \sum_{i=1}^4 (x_i x_{2i})^2 x_{4i}
\]
with indices considered modulo $5$. Then the quotient $X = Y /\mu_5$ has $\conFN(X) = 3$ where the $5$th-roots of unity $\zeta$ acts on $[x_1:x_2:x_3:x_4]$ with $\zeta^k$ on the homogeneous coordinate $x_k$.
\end{enumerate}
\end{propABC}

The example in \eqref{propitem:gentypeFN2} is an $8$-dimensional family of numerical Godeaux surfaces due to 
Miyaoka \cite[Theorem 5]{miyaoka_tricanonical_1976}, see also Reid \cite{reid_surfaces_1978}.

\medskip

In order to treat ruled surfaces, we prove a more general result about convex Fujita numbers of projective space bundles on curves.

\begin{thmABC}[see \tref{thm:FNonPEuppperbound} and \tref{thm:FNofPE Yusufs result}]
\label{thmABC:FNonPE}
Let $\cE$ be a vector bundle of rank $n$ and degree $d$ on a smooth projective curve $C$, and let $X = \bP(\cE)$ be the associated projective space bundle. Then we have
\[
n \leq \conFN(\bP(\cE)) \leq n+1.
\]
Furthermore, 
\begin{enumerate}[align=left,labelindent=0pt,leftmargin=*]
\item 
If $\cE$ is not stable then $\conFN(\bP(\cE)) = n$.
\item
Let $\cE$ be stable. Then the following holds.
\begin{enumerate}[align=left,labelindent=0pt,leftmargin=*]
\item 
If $n$ and $d$ are not coprime, then $\conFN(\bP(\cE)) = n$.
\item
If $d \equiv 1 \pmod n$, then $\conFN(\bP(\cE)) = n+1$. In particular, $\bP(\cE)$ is Fujita extreme in this case. 
\end{enumerate}
\end{enumerate}
\end{thmABC}

\tref{thmABC:FNonPE} relies on a criterion of Butler from \cite{butler_normal_1994} for ample line bundles on $\bP(\cE)$. Global sections of adjoint line bundles are analysed by computing the direct image along $\bP(\cE) \to C$ in combination with a well known slope criterion for global generation of semisimple vector bundles on curves, see \lref{lem:SlopeCriterionGloballyGeneratedVB}.

\subsection{Fujita extreme smooth projective surfaces} 

A smooth projective surface $X$ is called \textbf{Fujita extreme} if $\conFN(X) = 3$, the maximal possible value.
If we relax our search from being constrained to minimal surfaces, we find surprisingly many more examples with convex Fujita number $3$.

We recall that a group $\pi$ is called \textbf{projective} if it is isomorphic to the (topological) fundamental group $\pi_1(X)$ of a smooth projective variety $X$. 

\begin{thmABC}[see \pref{prop:cofinalFN3} and \tref{thm:allpi1FN3}]
\
\begin{enumerate}
\item
Let $X$ be a smooth projective surface. Then there is a birational modification $X' \to X$ such that $X'$ has convex Fujita number $\conFN(X') = 3$.
\item
For every projective group $\pi$ there is a smooth projective surface $X$ with $\pi_1(X) \simeq \pi$ and convex Fujita number $\conFN(X) =3$.
\end{enumerate}
\end{thmABC}

This theorem indicates that the topological invariant $\pi_1(X)$ alone is not sufficient to control positivity properties of adjoint line bundles.

\subsection*{Acknowledgements}
The authors acknowledge support by Deutsche  Forschungsgemeinschaft  (DFG) via  the Collaborative Research Centre TRR 326 "Geometry and Arithmetic of Uniformized Structures", project number 444845124.  Part of this work was done while the third author attended the workshop "Birational Complexity of Algebraic Varieties" at the Simons Center for Geometry and Physics, and he would like to thank the organizers and staff for the hospitality and stimulating atmosphere.

\section{Conventions and preliminaries}
\label{sec:prelims}

\subsection{Conventions}

We work over the complex numbers, although our results remain true over an arbitrary algebraically closed field of characteristic 0 by the Lefschetz principle. A surface is a variety (i.e.\ a separated scheme of finite type over $\bC$) of dimension two, which we assume to be smooth projective and connected without exception. 

\subsection{Revisiting Reider's method}
\label{sec:reider}
Convex Fujita numbers are finite by \cite[Proposition 2.5]{chen_convex_2023}. Reider's method gives precise bounds for surfaces as follows. 

\begin{prop}[{\cite[Proposition 2.1]{chen_convex_2023}}]
\label{prop:FNsurfacesReider}
Let $X$ be a smooth projective surface. 
\begin{enumerate}[align=left,labelindent=0pt,leftmargin=*]
\item 
\label{propitem:reider1}
The convex Fujita number of $X$ is bounded by $\conFN(X) \leq 3$.
\item 
\label{propitem:reider2}
If $\conFN(X) = 3$, then there exists an ample line bundle $\cL$ on $X$ with $\cL^2 = 1$.
\item 
\label{propitem:reider3}
If the intersection pairing on the N\'eron-Severi lattice $\NS(X)$ is even, then $\conFN(X) \leq 2$.
\item 
\label{propitem:reider4}
If the canonical bundle is numerically equal to $2 \vartheta$ with $\vartheta \in \Pic(X)$, then $\conFN(X) \leq 2$.
\end{enumerate}
\end{prop}

Let us revisit the argument in the proof of Reider's criterion for global generation. 
A point $P$ on the surface $X$ is a base point of the adjoint bundle $\omega_X \otimes \cL$ associated to an ample line bundle $\cL$ if and only if 
\begin{equation}
\label{eq:H1nonzero}
\rH^1(X,\omega_X \otimes \cL \otimes \cI_P) \not= 0\ .
\end{equation}
This follows from Kodaira vanishing and the long exact sequence
\[
\rH^0(X,\omega_X \otimes \cL) \to \rH^0(P,\omega_X \otimes \cL|_P) \to \rH^1(X,\omega_X \otimes \cL \otimes \cI_P) \to \rH^1(X,\omega_X \otimes \cL) = 0.
\]
By Serre duality, \eqref{eq:H1nonzero} is equivalent to the non-vanishing of $\Ext^1(\cL \otimes \cI_P, \dO_X)$.
So $P$ is a base point if and only if there exists a nontrivial extension
\begin{equation}
\label{eq:extension witness of base point}
0 \to \dO_X \xrightarrow{s} \cE \xrightarrow{t} \cL \otimes \cI_P \to 0\ ,
\end{equation}
and  the sheaf $\cE$ is necessarily a vector bundle by \cite[Proposition 1.33]{griffiths_residues_1978}.  For a detailed exposition of Reider's method we refer to \cite[\S3]{lazarsfeld_lectures_1997}.

Griffiths and Harris \cite{griffiths_residues_1978} describe \eqref{eq:extension witness of base point} more canonically as the Koszul resolution of the residue field at $P$ using $\cL = \det(\cE)$ as
\[
0 \to \dO_X \xrightarrow{s} \cE \xrightarrow{ \wedge s} \det(\cE)  \to  \det(\cE)  \otimes \dO_P \to 0. 
\]
This uses the fact that $P$ agrees with the vanishing locus $Z(s)$ of the regular section $s$ of $\cE$.

A concrete example of the above is described in the following proposition.

\begin{prop}
\label{prop:NCDampleFujitaExtremeSurface}
Let $X$ be a smooth projective surface with distinct ample effective divisors $C_i$, for $i=1, 2$ such that $(C_1 \bullet C_2) = 1$. Then the $C_i$ are irreducible and reduced, intersect in a single point $P = C_1 \cap C_2$, and $P$ is a base point for $\omega_X \otimes \dO_X(C_1 + C_2)$.  In particular, the convex Fujita number of $X$ is 
\[
\conFN(X) = 3.
\]
\end{prop}
\begin{proof}
Since $C_1$ is ample, the intersection number with each component of $C_2$ is positive. It follows that $C_2$ is irreducible and reduced. By symmetry the same holds for $C_1$. Since $C_1$ and $C_2$ are distinct, they intersect in a scheme $Z = C_1 \cap C_2$ of dimension $0$ such that the length of $\dO_Z$ equal to $C_1 \bullet C_2 = 1$. It follows that $Z = P$ is a single reduced point. Moreover, the sequence
\[
0 \to \dO_X \xrightarrow{s_1,s_2}  \dO_X(C_1) \oplus \dO_X(C_2) \xrightarrow{(s_2,-s_1) \bullet} \dO_X(C_1 + C_2) \otimes \cI_P \to 0
\]
is exact. Here $(s_2,-s_1) \bullet$ is the map that sends a pair $(t_1,t_2)$ to $s_2t_1 - s_1 t_2$.  Exactness follows because the defining equations of $C_1$ and $C_2$ at $P$ form a regular sequence of  parameters at $P$.
The extension is non-trivial because $\dO_X(C_1 + C_2) \otimes \cI_P$ is not a vector bundle. 
The discussion preceeding \pref{prop:NCDampleFujitaExtremeSurface} shows that $P$ is a base point of $\omega_X \otimes \dO_X(C_1+C_2)$. 
This means that $\conFN(X) \geq 3$. The converse inequality is 
\pref{prop:FNsurfacesReider}\eqref{propitem:reider1}.
\end{proof}

\subsection{Fujita numbers of products}

The following useful result was proven in 
\cite[Lemma 2.6 and Proposition 2.7]{chen_convex_2023}.

\begin{prop}
\label{prop:FNforProduct}
Let $X$ and $Y$ be smooth projective varieties. Then we have 
\[
\conFN(X \times Y) \geq \max\{\conFN(X), \conFN(Y) \},
\]
with equality if the abelian varieties $\Pic^0_X$ and $\Pic^0_Y$ have no common nontrivial isogeny factor.
\end{prop}

As a corollary we can determine the convex Fujita numbers of a product of curves.

\begin{cor}
\label{cor:FNproduct of curves}
The convex Fujita number of a product $X=C_1 \times C_2$ of two smooth projective curves equals
\[
\conFN(C_1 \times C_2)  =2.
\]
\end{cor}
\begin{proof}
Since curves have convex Fujita number $2$, \pref{prop:FNforProduct} shows that $\conFN(C_1 \times C_2) \geq 2$. On the other hand, the canonical bundle  $\omega_X = \omega_{C_1} \boxtimes \omega_{C_2}$ is divisible by $2$, so that the upper bound follows from \pref{prop:FNsurfacesReider}\eqref{propitem:reider4}.
\end{proof}

\begin{cor}
\label{cor:productsurfaces FN2}
There are minimal surfaces of convex Fujita number $2$ of the following kind:
\begin{enumerate}[align=left,labelindent=0pt,leftmargin=*]
\item 
rational surfaces,
\item 
ruled surfaces over an arbitrary smooth projective curve $S$ as a base,
\item
abelian surfaces,
\item 
elliptic surfaces of Kodaira dimension $1$,
\item
surfaces of general type.
\end{enumerate}
\end{cor}
\begin{proof}
All of these are realized as products of curves: 
\[
\bP^1 \times \bP^1, \ \bP^1 \times S, \ E \times E', \ E \times C, \ C_1 \times C_2,
\]
where $E$ and $E'$ are elliptic curves, and $C$, $C_1$ and $C_2$ are curves of genus at least $2$.
\end{proof}

\subsection{Pseudosplit irreducible fibrations}

We are going to describe a geometric setting which leads to an ample line bundle such that the adjoint linear system has base points. 

\begin{lem}
\label{lem:FNleq1impliesintersectiongeq2}
Let $X$ be a smooth projective surface with $\conFN(X) \leq 1$. Let $D$ be an ample effective divisor on $X$, and let $i: C \inj X$ be an irreducible and reduced nef divisor on $X$. Then we have $(D \bullet C) \geq 2$.
\end{lem}
\begin{proof}
We argue by contradiction and assume that $(D \bullet C) = 1$, i.e., $D$ intersects $C$ transversely in a unique smooth point $P$. The divisor $C+D$ is also ample, so by assumption $\omega_X \otimes \dO_X(C+D)$ is globally generated. Therefore its restriction to $C$ is also globally generated. 
By adjunction (using the resolution $\dO_X(-C) \to \dO_X$ of $i_\ast \dO_C$) 
\[
\omega_C = \cExt^1(i_\ast \dO_C, \omega_X) = \coker(\omega_X \to \omega_X(C)) = \omega_X(C)|_C
\]
where $\omega_C$ is the dualizing sheaf of $C$,  we compute the restriction as
\[
(\omega_X \otimes \dO_X(C+D))|_C = \omega_X(C)|_C \otimes \dO_C(P) = \omega_C(P).
\]
The cohomology sequence 
\[
0 \to \rH^0(C,\omega_C) \to \rH^0(C,\omega_C(P)) \to \rH^0(C,\omega_C(P)|_P) \to \rH^1(C,\omega_C) \to \rH^1(C,\omega_C(P)) \to 0
\]
shows that $P$ is a base point of $\omega_C(P)$  if and only if 
\[
\rH^1(C,\omega_C) \to \rH^1(C,\omega_C(P)) 
\]
is not an isomorphism. Serre duality translates this into the dual map 
\[
 0 = \rH^0(C,\dO_C(-P)) \to \rH^0(C,\dO_C) \not= 0\ ,
\]
which is \emph{not} an isomorphism. This shows that $P$ is a base point of $\omega_X(C+D)$, a contradiction.
\end{proof}

\begin{defi}
\label{defi:special fibration}
A \textbf{pseudosplit irreducible fibration} on a smooth projective surface $X$ is a fibration $f: X \to B$ over a smooth curve $B$  such that 
\begin{enumerate}[label=(\roman*),align=left,labelindent=0pt,leftmargin=*]
\item 
all fibres $X_b = f^{-1}(b)$ are irreducible, and
\item 
there is a fibre $X_b$, potentially multiple, with underlying reduced fibre $F$, and 
\item 
an irreducible curve $S$ on $X$ with $S \bullet  F = 1$.
\end{enumerate}
\end{defi}

\begin{prop}
\label{prop:specialfibrationFN}
Let $X$ be a smooth projective surface that admits a pseudosplit irreducible fibration $f: X \to B$ with reduced fibre $F$ of multiplicity $m$ and irreducible curve $S$ such that $S \bullet F  = 1$. Let $D$ be a nontrivial effective divisor on $B$ of degree $\deg(D) > -\frac{1}{m}(S^2)$. Then  
\[
\cL = \dO_X(S + F + f^\ast D)
\]
is ample, but  $\omega_X \otimes \cL$ is not globally generated. More precisely, the intersection point $P = S \cap F$ is contained in the base locus of $\omega_X \otimes \cL$. 
\end{prop}
\begin{proof}
We set $\cM = \dO_X(S + f^\ast D)$ so that $\cL = \cM(F)$.  We first show that $\cM$ is ample. 
We have $\cM \bullet F = 1$ and $\cM \bullet S = (S^2) + \deg(D) m > 0$. Furthermore we  also have
\[
(\cM^2) = (S^2) + 2\deg(D)m  > 0.
\]
For any irreducible curve $C$ on $X$ other than $S$ or $F$ we have $(S \bullet C)  \geq 0$ and $(f^\ast D \bullet C) \geq 0$ with at least one intersection number strictly positive, because $D$ is effective and because all fibres of $f: X \to B$ are irreducible by assumption. It follows that $(\cM \bullet C) > 0$, and therefore $\cM$ is ample by the criterion of Nakai-Moishezon. 

Since $F$ is nef and $\cM \bullet F = 1$ and $\cM$ is ample effective, we can apply \lref{lem:FNleq1impliesintersectiongeq2} to deduce that $\conFN(X) \geq 2$. More precisely, the proof of \lref{lem:FNleq1impliesintersectiongeq2} shows that $P = S \cap F$ is a base point of $\omega_X \otimes \cL$ as claimed. 
\end{proof}

\begin{cor}
\label{cor:specialfibrationFN}
Let $X$ be a smooth projective surface that admits a pseudosplit irreducible fibration. Then we have $\conFN(X) \geq 2$.
\end{cor}
\begin{proof}
This follows at once from \pref{prop:specialfibrationFN} by choosing an effective divisor on the base of the pseudosplit irreducible fibration of  suitably large degree. 
\end{proof}

\section{Projective bundles on curves} 
\label{sec:FujitaPE}

In this section we study convex Fujita numbers for projective bundles on curves. In a forthcoming work we will treat projective space bundles on more general varieties. 

\subsection{Review of ample line bundles on projective space bundles} 
Let $\cE$ be a vector bundle of rank $r$ on a smooth projective variety $S$. We will study the convex Fujita number of $X = \bP(\cE)$, the associated projective space bundle $\pi: \bP(\cE) \to S$. The Picard group of $\bP(\cE)$ sits in a short exact sequence
\[
0 \to \Pic(S) \xrightarrow{\pi^\ast} \Pic(\bP(\cE)) \xrightarrow{\res} \Pic(\bP^{r-1}) \to 0
\]
where $\res$ restricts to a fiber and $\Pic(\bP^{r-1} )= \bZ$. The sequence splits using the tautological line bundle $\dO(1)$ associated to $\cE$. Every line bundle on $\bP(\cE)$ is uniquely of the form $\pi^\ast \cM(a)$.

We start with a lower bound for the convex Fujita number.

\begin{prop}
\label{prop:FNonPElowerbound}
Let $\cE$ be a vector bundle of rank $r$ on a smooth projective variety $S$. Then 
\[
\conFN(\bP(\cE)) \geq \conFN(\bP^{r-1}) = r.
\]
\end{prop}
\begin{proof}
Let $\pi : X=\bP(\cE) \to S$ be the natural  projection. 
Let $\cM$ be an ample line bundle on $S$ such that $\cM \otimes \cE$ is globally generated. Then $\cL = \pi^\ast \cM^{\otimes 2} (1)$ is ample on $X=\bP(\cE)$. Indeed, given a tuple of $n+1$ generating sections $\dO_S^{n+1} \surj \cM \otimes \cE$, the associated closed immersion 
\[
X = \bP(\cE) = \bP(\cM \otimes \cE) \inj \bP(\dO^{n+1}) = \bP^n \times S
\]
has $\pi^\ast \cM^{\otimes 2} (1)$ isomorphic to the restriction of $\dO(1) \boxtimes \cM$ to $X$, which is ample. 

If $\omega_X \otimes \cL^{\otimes m}$ is globally generated, then its restriction to a fiber (identified with $\bP^{r-1}$) is also globally generated and equals
\[
(\omega_X \otimes \cL^{\otimes m})|_{\bP^{r-1}} \simeq \omega_{\bP^{r-1}} \otimes \dO(m) \simeq \dO(m - r)\ .
\]
It follows that  $m \geq r$.
\end{proof}

Recall that a vector bundle $\cE$ on a smooth projective curve $C$ admits a unique Harder-Narasimhan filtration
\[
0 = \cE_0 \subseteq \cE_1 \subseteq \cE_2 \subseteq \ldots \subseteq \cE_{t-1} \subseteq \cE_t = \cE
\]
by  vector bundles $\cE_i$ such that $\cE_i/\cE_{i-1}$ is semistable of slope $\mu_i$ with 
\[
\mu^+(\cE) \coloneq \mu_1 > \mu_2 > \ldots > \mu_{t-1} > \mu_t \eqcolon \mu^-(\cE).
\]
Since the slope of an extension is always in the interval of the slopes of its constituents, we find
\[
\mu^-(\cE) \leq \mu(\cE) \leq \mu^+(\cE)
\]
with equality if and only if $\cE$ is semistable. We call $\mu^+(\cE)$ (resp.\ $\mu^-(\cE)$) the \textbf{maximal} (resp.\ \textbf{minimal}) \textbf{slope} of $\cE$. Butler deduced the following observation from  \cite[Theorem 3.1]{miyaoka_chern_1987}.

\begin{prop}
\label{prop:ample cone on PE over curve}
Let $\pi:  \bP(\cE) \to C$ be the projective space bundle of vector bundle $\cE$ of rank $r$ on a smooth projective  curve $C$. A line bundle $\cL = \pi^\ast \cM (a)$ is ample if and only if $a > 0$ and 
\[
\deg(\cM) + a \mu^-(\cE) > 0.
\]
\end{prop}
\begin{proof}
\cite[Lemma 5.4]{butler_normal_1994}.
\end{proof}

\subsection{Computation of convex Fujita numbers}

The following theorem provides examples of  (almost) Fujita extreme varieties.

\begin{thm}
\label{thm:FNonPEuppperbound}
Let $\cE$ be a vector bundle of rank $n$ on a smooth projective curve $C$. 
Then we have
\[
n \leq \conFN(\bP(\cE)) \leq n+1.
\]
If $\cE$ is not semistable then $\conFN(\bP(\cE)) = n$. 
\end{thm}
\begin{proof}
We set $X = \bP(\cE)$ and write $\pi: X \to C$ for the natural projection map.
The lower bound was established in \pref{prop:FNonPElowerbound}.  For the upper bound let $s \geq n$ and let $\cL_i = \pi^\ast\cM_i(a_i)$ for $i=1, \ldots, s$ be ample line bundles on $X$. We abbreviate $a \coloneq \sum_{i=1}^s a_i$, and $\cM \coloneq \bigotimes_{i=1}^s \cM_i$ and $\cL \coloneq \bigotimes_{i=1}^s \cL_i = \pi^\ast\cM (a)$. By \pref{prop:ample cone on PE over curve} we have $a_i \geq 1$ and thus $a \geq s \geq n$.  Recall that we have 
\[
\omega_X = \dO(-n) \otimes \pi^\ast(\omega_C \otimes \det(\cE)).
\]
Fiberwise, $\omega_X \otimes \cL$ restricts to $\dO(a-n)$ and thus is globally generated with trivial higher cohomology. Cohomology and base change shows the surjectivity of the 'relative global generation' map 
\[
\pi^\ast \pi_\ast(\omega_X \otimes \cL) \surj \omega_X \otimes \cL\ ,
\]
because for $t \in C$ with fiber $X_t = \pi^{-1}(t)$ we have
\begin{align*}
\pi^\ast \pi_\ast(\omega_X \otimes \cL)|_{X_t} & = \dO_{X_t} \otimes \big(\pi_\ast(\omega_X \otimes \cL)|_t\big) =  \dO_{X_t} \otimes \rH^0\big(X_t, (\omega_X \otimes \cL)|_{X_t}\big) \\
& = \dO_{X_t} \otimes \rH^0\big(X_t,\dO(a-n)\big) \surj  \dO(a-n) = (\omega_X \otimes \cL)|_{X_t}\ .
\end{align*}
It follows that $\omega_X \otimes \cL$ is globally generated if $\pi_\ast(\omega_X \otimes \cL)$ is globally generated. By the projection formula, we have
\[
\pi_\ast(\omega_X \otimes \cL) = \omega_C \otimes \Sym^{a-n}(\cE) \otimes \det(\cE) \otimes \cM =: \omega_C \otimes \cF,
\] 
where we abbreviate  $\cF \coloneq \Sym^{a-n}(\cE) \otimes \det(\cE) \otimes \cM$.

Let $\cE^-$ be the quotient of $\cE$ which is the part of its Harder-Narasimhan filtration with minimal slope. Since symmetric powers and tensor products of semistable vector bundles are again semistable, see  \cite[Corollary 3.7 and 3.10]{miyaoka_chern_1987}, we find 
\[
\cF^-  = \Sym^{a-n}(\cE^-) \otimes \det(\cE) \otimes \cM.
\]
Let $n^-$ be the rank of $\cE^-$. 
By \pref{prop:ample cone on PE over curve}, we compute the minimal slope of $\cF$ as
\begin{align*}
\mu^-(\cF) & = \mu(\cF^-)  = \mu(\Sym^{a-n}(\cE^-)) + \mu(\det(\cE)) + \mu(\cM)  \\
& = (a-n) \mu^-(\cE) + n \mu(\cE) + \mu(\cM)  \\ 
& =  n  \big( \mu(\cE) - \mu^-(\cE)\big) + \sum_{i=1}^s \big(\deg(\cM_i) + a_i \mu^-(\cE) \big) 
\geq  \frac{s}{n^-} \geq 1.
\end{align*}
If either $\cE$ is not semistable or $s \geq n+1$, we have $s > n^-$ and so  $\mu^-(\cF) > 1$. By the well known \lref{lem:SlopeCriterionGloballyGeneratedVB} below, the adjoint bundle $\omega_C \otimes \cF$ is globally generated, and the proof is complete.
\end{proof}

\begin{lem}
\label{lem:SlopeCriterionGloballyGeneratedVB}
Let $C$ be a smooth projective curve, and let $\cF$ be a vector bundle on $C$ of minimal slope $\mu^-(\cF) > 1$. Then $\omega_C \otimes \cF$ is globally generated. 
\end{lem}
\begin{proof}
Let $P$ be a point in $C$. Global sections of $\omega_C \otimes \cF$ generate in $P$ if $\rH^1(C,\omega_C \otimes \cF(-P)) = 0$. By Serre duality, it suffices to show $\rH^0(C,\cF^\vee(P)) = 0$. The maximal slope of these dual coefficients is
\[
\mu^+(\cF^\vee(P)) = 1 - \mu^-(\cF) < 0.
\]
Any nontrivial map $\dO_C \to \cF^\vee(P)$ therefore violates the semistability of the filtration quotients of the Harder-Narasimhan filtration of $\cF^\vee(P)$. 
\end{proof}

Under suitable hypotheses we are able to determine  convex Fujita numbers of projectivized semistable vector bundles completely.

\begin{thm}
\label{thm:FNofPE Yusufs result}
Let $\cE$ be a semistable vector bundle of rank $n$ and degree $d$ on a smooth projective curve $C$. 
\begin{enumerate}[align=left,labelindent=0pt,leftmargin=*]
\item \label{propitem:FNonPEoverCunstable}
If $\cE$ is not stable, then $\conFN(\bP(\cE)) = n$.
\item 
\label{propitem:FNonPEoverC1}
If $n$ and $d$ are not coprime, then $\conFN(\bP(\cE)) = n$.
\item
\label{propitem:FNonPEoverC2}
 If $n$ and $d$ are coprime, then the following holds. 
\begin{enumerate}[label=(\roman*),align=left,labelindent=0pt,leftmargin=*]
\item 
\label{propitem:FNonPEoverC2i} 
Let $C$ have genus $\geq 2$. If $d \not\equiv 1 \pmod n$ and $\Sym^{nk}(\cE)$ has no direct summand that is a line bundle for all $k > 0$, then $\conFN(\bP(\cE)) = n$.
\item 
\label{propitem:FNonPEoverC2ii} 
If $d \equiv 1 \pmod n$, then $\conFN(\bP(\cE)) = n+1$.  In particular, $\bP(\cE)$ is Fujita extreme in this case. 
\end{enumerate}
\end{enumerate}
\end{thm}
\begin{proof}
We start with assertion~\eqref{propitem:FNonPEoverC1} and keep the notation of \tref{thm:FNonPEuppperbound}.
We need to discuss the case $s = n$ for a semistable $\cE$ and decide whether there are $\cM_i$ and $a_i$  such that $\omega_C \otimes \cF = \pi_\ast(\omega_X \otimes \cL)$ is not globally generated. Note that $\cF$ is now also semistable and so $\mu(\cF) = \mu^-(\cF)$. The proof of \tref{thm:FNonPEuppperbound} shows that the only critical case is when $\mu(\cF) = 1$, and that happens exactly when 
\begin{equation}
\label{eq:criticalcasePEonCurve}
n \deg(\cM_i) + a_i d = 1, \text{ for all $i = 1, \ldots, n$},
\end{equation}
This is only possible if $n$ and $d$ are coprime. This shows assertion \eqref{propitem:FNonPEoverC1}.

\smallskip

Assertion~\eqref{propitem:FNonPEoverCunstable} follows immediately from \eqref{propitem:FNonPEoverC1} since the slope $\mu = d/n$ of an unstable semistable $\cE$ is also the slope of a vector bundle with rank $< n$, so $d$ and $n$ must have a common prime factor. Assertion \eqref{propitem:FNonPEoverCunstable} appears here because it proves part of \tref{thmABC:FNonPE}.

\smallskip

We now assume $n$ and $d$ are coprime and prove assertion~\eqref{propitem:FNonPEoverC2}. Then, as just recalled, $\cE$ is stable. Moreover, we assume that $\cM_i$ and $a_i \geq 1$ are such that the equations \eqref{eq:criticalcasePEonCurve} hold. It follows that all the $a_i$ are congruent to each other modulo $n$. Thus $a-n$ is divisible by $n$.

For assertion \ref{propitem:FNonPEoverC2i} we now also assume that $d \not\equiv 1 \pmod n$. Then $a_i \geq 2$ and $a-n \geq n$. Recall from the proof of \tref{thm:FNonPEuppperbound} that it suffices to show that $\omega_C \otimes \cF$ is globally generated. The proof of \lref{lem:SlopeCriterionGloballyGeneratedVB} shows that it suffices that for all points $P \in C$ the sheaf 
\[
\cF^\vee(P) \simeq \Sym^{a-n}(\cE^\vee) \otimes \det(\cE)^{-1} \otimes \cM^{-1}
\]
has no global sections.  By Narasimhan--Seshadri the stable $\cE$ is described by a unitary representation of a central extension of $\pi_1(C)$. Thus, $\Sym^{a-n}(\cE^\vee)$ corresponds  to a direct sum of unitary representations and therefore is polystable as a vector bundle. The same holds for $\cF^\vee(P)$. By assumption, none of the direct summands is a line bundle (note that $n \mid a-n$). By the slope computation from above, $ 
\cF^\vee(P)$ is a polystable vector bundle of slope $0$, and so any global section yields a direct summand of the form $\dO_C$, a contradiction. This shows  \ref{propitem:FNonPEoverC2i}.

We now prove assertion \ref{propitem:FNonPEoverC2ii}. By assumption $d = 1- kn$ for some $k \in \bZ$. Let $\cM_i$ be line bundles on $C$ of degree $k$, and let $a_i$ be equal to $1$ for all $i = 1, \ldots, n$. Then \eqref{eq:criticalcasePEonCurve} holds, and in particular $\cL_i = \pi^\ast \cM_i(1)$ is ample. More precisely we have
\[
\omega_X \otimes \cL = \dO(-n) \otimes \pi^\ast(\omega_C \otimes \det(\cE)) \otimes \pi^\ast \cM (n) = \pi^\ast(\omega_C \otimes \det(\cE) \otimes \cM).
\]
It follows that $\omega_X \otimes \cL$ is globally generated if and only if $\omega_C \otimes \cF = \omega_C \otimes \det(\cE) \otimes \cM$ is globally generated on $C$. Since $\mu(\cF) = 1$, by an appropriate choice of $\cM_i$ we may find $\cF = \dO_C (P)$ for a point $P \in C$. In this case $\omega_C \otimes \cF$ has $P$ as a base point, and this concludes the proof of assertion \ref{propitem:FNonPEoverC2ii}.
\end{proof}

\begin{rmk}
Semistable vector bundles of rank $n$ and degree $d$ on a curve $C$ of genus $\geq 2$ exist by 
\cite[Lemma 4.3]{narasimhan_moduli_1969}. When $d$ and $n$ are coprime, then the semistable vector bundle $\cE$ is in fact stable. If $\cE$ is generic, then $\Sym^k(\cE)$ is again stable for all $k \geq 0$ by a result of Seshadri, 
see Hartshorne \cite[Theorem 10.5]{hartshorne_ample_1970}. So vector bundles satisfying the respective assumptions of \tref{thm:FNofPE Yusufs result} exist in abundance. 

Balaji and Koll\'ar define a notion of a holonomy group for a stable vector bundle $\cE$ on an arbitrary smooth projective variety $S$. This is a reductive subgroup of the automorphism group $\GL(\cE(x))$ of a fiber $\cE(x) := \cE \otimes \kappa(x)$ which is minimal to contain all Narasimhan-Seshadri representations group associated to $\cE|_C$ for curves $x \in C \subseteq S$, as long as $\cE|_C$ is still stable.  In \cite[Corollary 6]{balaji_holonomy_2008} they prove that if the commutator subgroup of the holonomy group is either $\SL(\cE(x))$ or $\Sp(\cE(x))$, then $\Sym^k(\cE)$ is still stable for all $k \geq 0$. So with this assumption on holonomy of $\cE$  the assumptions of the case \ref{propitem:FNonPEoverC2i}  in \tref{thm:FNofPE Yusufs result}  holds provided $d \not\equiv 1 \pmod n$. 
\end{rmk}

\section{Rational and ruled surfaces}
\label{sec:FN rational ruled}

In this section we start our study of  how convex Fujita numbers vary across the Kodaira--Enriques classification of surfaces. In negative Kodaira dimension, minimal surfaces are rational or ruled.

\subsection{Ruled surfaces}
We start with ruled surfaces since  the result follows at once from the general discussion in \secref{sec:FujitaPE}.

\begin{prop}
\label{prop:FNruled surface}
Let $C$ be a smooth projective curve and let $\cE$ be a vector bundle of rank $2$ on $C$.
The convex Fujita number of the ruled surface $X=\bP(\cE) \to C$ equals
\[
\conFN(X) = \begin{cases}
3 & \text{ if $\cE$ is stable of odd degree,} \\
2 & \text{ else}.
\end{cases}
\]
\end{prop}
\begin{proof}
This is just a special case of \tref{thm:FNonPEuppperbound} and \tref{thm:FNofPE Yusufs result}.
\end{proof}

\subsection{Rational surfaces}

A relatively minimal rational surfaces is either the  projective plane or a Hirzebruch surface.

\begin{prop}
\label{prop:FNrational surface}
The convex Fujita number of a relatively minimal rational surface is $2$ or $3$. More precisely,
\begin{enumerate}[align=left,labelindent=0pt,leftmargin=*]
\item
\label{propitem:FNrational surfaceP2}
 we have $\conFN(\bP^2) = 3$, and 
\item 
\label{propitem:FNrational surface Hirzebruch}
the convex Fujita number of a Hirzebruch surface equals $2$.
\end{enumerate}
\end{prop}
\begin{proof}
Assertion \eqref{propitem:FNrational surfaceP2} follows, because an ample line bundle on $\bP^2$ is isomorphic to $\dO(a)$ with $a > 0$, and so $\omega_{\bP^2} \otimes \dO(a_1) \otimes \ldots \otimes \dO(a_s)$ is globally generated if and only if $\sum_i a_i \geq 3$. In the worst case all $a_i = 1$, and thus we need $s \geq 3$.

\eqref{propitem:FNrational surface Hirzebruch}
Let $X = \bP(\dO \oplus \dO(n))$ be a Hirzebruch surface.  \pref{prop:FNruled surface} yields $\conFN(X)$, because vector bundles of rank $2$ on $\bP^1$ are never stable. 
\end{proof}

\begin{rmk}
Since Hirzebruch surfaces are very explicit,  convex Fujita numbers can also be computed directly. We give an elementary argument  to illustrate this fact.

Since $X = \bP(\dO \oplus \dO(n))$ is a toric surface, nef and globally generated are equivalent for all line bundles on $X$, see \cite[Theorem 3.1]{mustata_vanishing_2002}. The Picard group $\Pic(X) = \NS(X)$ is generated by the class of a fiber $F$ and the class of a section $S$ with self intersection $S^2 = -n$. The class $L = aS + bF$ is nef and equivalently globally generated if and only if $b \geq na \geq 0$, and the ample classes are characterized by $b > na > 0$. Recall that the canonical class is $K = -2S - (n+2)F$. 

For line bundles $L_i = a_i S + b_i F$, for $i = 1, \ldots m$, we find
\[
K + \sum_{i=1}^m L_i = (-2 + \sum_{i=1}^m a_i )S + (-n-2 + \sum_{i=1}^m b_i ) F\ .
\] 
For the ample line bundles $L_i = S + (n+1)F$ this becomes 
\[
K + \sum_{i=1}^m L_i = (m-2)S + ((m-1)(n+1)-1)F\ ,
\]
which  is globally generated if and only if  $m \geq 2$. When $m \geq 2$ and  $L_i$ an ample line bundle, we have  
\[
-n-2 + \sum_{i=1}^m b_i \geq -n-2 + \sum_{i =1}^m (na_i+1) \geq  n \big(- 2  + \sum_{i=1}^m a_i\big) \geq 0\ ,
\]
and so the corresponding adjoint bundle is globally generated. This yields $\conFN(X) = 2$.
\end{rmk}

\section{Kodaira dimension \texorpdfstring{$0$}{0}}
\label{sec:FN Kodaira 0}

On a minimal surface $X$ of Kodaira dimension $0$ the canonical class is numerically trivial. It follows from \pref{prop:FNsurfacesReider} that $\conFN(X) \leq 2$. The canonical class is globally generated for abelian surfaces and for K3 surfaces, but is a nontrivial torsion class hence  not globally generated for Enriques  and  bielliptic surfaces. Therefore,  the options for  $\conFN(X)$ of abelian surfaces and K3 surfaces are $0$ or $2$, while for Enriques and  bielliptic surfaces the options are $1$ or $2$. 

\begin{rmk}
Before we present our results, we recall that Reider in \cite[Proposition 5]{reider_vector_1988} reports on the following result of Beauville: let $X$ be of Kodaira dimension $0$ and $\cL$ a nef line bundle on $X$ with $(\cL^2) \geq 6$. Then $\omega_X \otimes \cL$ is globally generated precisely  if $X$ does not contain a $1$-connected effective cycle $E$ with arithmetic genus $0$ and $E \bullet \cL = 1$.

Our results restrict to ample line bundles and seek to determine the convex Fujita number of $X$. Clearly, since our main tool is Reider's method, there is bound to be  some overlap between our analysis and Beauville's.
\end{rmk}

\subsection{Abelian surfaces}

The convex Fujita number of an abelian variety is always at most two by a generalization of Lefschetz' classical theorem by 
Bauer and Szemberg \cite[Theorem 1.1]{bauer_tensor_1996} (cf. \cite[Example 1.8]{chen_convex_2023}).
Here we deal with abelian surfaces and determine the arising  convex Fujita numbers precisely.

\begin{prop}
\label{prop:globalgeneration abelian surface}
Let $\cL$ be an ample line bundle on an abelian surface $X$. 
\begin{enumerate}[align=left,labelindent=0pt,leftmargin=*]
\item
\label{propitem:globalgeneration abelian surface 1}
If $(\cL^2) \leq 4$, then $\cL$ is not globally generated.
\item
\label{propitem:globalgeneration abelian surface 2}
If $(\cL^2) > 4$, then $\cL$ is globally generated unless 
\begin{itemize}
\item
$X \simeq E \times E'$ is isomorphic to a product of elliptic curves $E$ and $E'$, and 
\item
$\cL$ is isomorphic to $\cM \boxtimes \cM'$ with $\cM$ of degree $1$ on $E$ and $\cM'$ of positive degree on $E'$.
\end{itemize}
\end{enumerate}
\end{prop}
\begin{proof}
By Kodaira vanishing and Riemann--Roch, $\rh^0(X,\cL) = (\cL^2)/2 \leq 2$.  If $\cL$ were globally generated, then the associated map $X \to \bP(\rH^0(X,\cL))$ would be  finite as  $\cL$ is ample. This yields a contradiction if $(\cL^2) \leq 4$ by comparing dimensions. This shows \eqref{propitem:globalgeneration abelian surface 1}.

Moving on to \eqref{propitem:globalgeneration abelian surface 2}, we assume $(\cL^2) > 4$ (more precisely $(\cL^2) \geq 6$ as the intersection form is even). Reider's method shows that an ample $\cL$ with $(\cL^2) \geq 5$ is globally generated unless there is an effective divisor $C$ on $X$ with $C^2=0$ and $C\bullet \cL =1$. Since $\cL$ is ample, we deduce that $C$ must be irreducible and reduced. Riemann--Roch computes $\chi(C,\dO_C) = 0$, and since abelian varieties do not contain rational curves, the curve $C$ is smooth of genus $1$. After translation, the curve $C$ is an abelian subvariety of $X$ that from now on we denote by $E$. The quotient map $X \to X/E \eqcolon E'$ realizes $X$ as an elliptic fibration over an elliptic curve $E'$. 

The degree of $\cL$ equals $1$ on a general and thus on the generic fiber $X_\eta$. By Riemann--Roch on the curve $X_\eta$, there is a point $P \in X_\eta$ of degree $1$ with $\cL|_{X_\eta} \simeq \dO_{X_\eta}(P)$. The Zariski closure of $P$ in $X$ yields a splitting $s: E' \to X$ of the quotient map $X \to E'$ and an isomorphism $X \simeq E \times E'$. Note that, by construction, the line bundles $\cL$ and $\dO_X(s(E')) = \pr^\ast \dO_E(P_0)$ with the origin $P_0 \in E$ agree on $X_\eta$ (here $\pr: X \to E$ is the projection). We set $\cM = \dO_E(P_0)$ and conclude that the difference $\cL \otimes \pr^\ast \cM^{-1}$ comes from a line bundle $\cM'$ on $E'$ with thus $\cL \simeq \cM \boxtimes \cM'$.  

It remains to show that in the exceptional case constructed above the line bundle $\cL$ is not globally generated. But $\cL|_E$ has degree $1$ and thus is not globally generated on $E$. Hence $\cL$ has a base point along $E$. 
(This is a special case of \pref{prop:specialfibrationFN} because $\omega_X$ is trivial, the ample line bundle $\cL = \dO_X(S)$ is effective and $X \to E'$ is a pseudosplit irreducible fibration since $S \bullet E = 1$ for the fibre $E$.)
\end{proof}

\begin{prop}
\label{prop:FNabeliansurface}
Let $X$ be an abelian surface. Then we have
\[
\conFN(X) = 2 \quad \iff \quad \text{ $X$ supports  an ample $\cL$ with $(\cL^2) \leq 4$}\ .
\]
If the above assertions do not hold, then we have $\conFN(X) = 0$. Both values $0$ and $2$ occur:
\begin{enumerate}[align=left,labelindent=0pt,leftmargin=*]
\item 
\label{propitem:FNabeliansurface FN2}
If $X$ is a principally polarized abelian variety, then $\conFN(X)$ equals $2$.
\item 
\label{propitem:FNabeliansurface FN0}
If $X$ is an abelian variety with an isogeny $A \to X \coloneq A/G$ for a principally polarized abelian variety $A$ with $\End(A) = \bZ$ and a finite group $G$ whose odd part of $\#G$ is not a square. Then $\conFN(X)$ equals $0$.
\end{enumerate}
\end{prop}
\begin{proof}
We know that $\conFN(X)$ is either $0$ or $2$, and that it is $0$ if and only if all ample line bundles are globally generated. The equivalence thus follows at once from \pref{prop:globalgeneration abelian surface}. If there is an ample line bundle with $(\cL^2) \leq 4$ then this $\cL$ is not globally generated, hence  $\conFN(X) = 2$. If, on the other hand, for all ample line bundles $(\cL^2) \geq 6$ then $X$ is not a product of elliptic curves, as on such a product we have ample line bundles of self intersection $2$. Then the exceptional case of \pref{prop:globalgeneration abelian surface} does not occur, and all ample $\cL$ are globally generated. This shows  $\conFN(X) = 0$.

\eqref{propitem:FNabeliansurface FN2} is obvious because a principal polarization $\cL$ has $(\cL^2) = 2$. 
It remains to prove \eqref{propitem:FNabeliansurface FN0}.
Let $\pi: A \to X$ denote the quotient map. For an ample line bundle $\cL$ on $X$ we denote by $\ph_\cL : X \to X^t$ the isogeny $\ph_\cL(x) = t_x^\ast \cL \otimes \cL^{-1}$. The composition
\[
A \xrightarrow{\pi} X \xrightarrow{\ph_\cL} X^t \xrightarrow{\pi^t} A^t \simeq A
\]
is multiplication by an integer $m \in \bZ$ because $\End(A) = \bZ$. Computing degrees yields
\[
m^4 = \deg(\pi) \cdot \deg(\pi^t) \cdot \deg(\ph_\cL) = \#G^2 \cdot \deg(\ph_\cL)\ .
\]
By Riemann--Roch $\deg(\ph_\cL) = \chi(A,\cL)^2 = \big(\frac{1}{2} (\cL^2)\big)^2$. Combining the two equations yields
\[
2 m^2 = \#G \cdot  (\cL^2)\ .
\]
We now choose an odd prime $p$ that occurs in $\#G$ with an odd exponent. Since the self-intersection $(\cL^2)$ is always even due to Riemann--Roch, it follows that  $2p \mid (\cL^2)$. Since then $(\cL^2) \geq 6$, we are done. 
\end{proof}

\begin{rmk}
\label{rmk:FN in etale cover}
The abelian surface $X$ constructed in \pref{prop:FNabeliansurface} \eqref{propitem:FNabeliansurface FN0}
 together with the isogeny $\pi: A \to X$ with a principally polarized abelian surface $A$ give rise to finite \'etale maps
 \[
 A \xrightarrow{\pi} X \to A
 \]
(as in the proof of \pref{prop:FNabeliansurface}), showing that the convex Fujita number can actually go up and down along finite \'etale maps. 
\end{rmk}

\subsection{Bielliptic surfaces}

Next we consider bielliptic surfaces $X = E \times F/G$ where $E$ and $F$ are elliptic curves and $G$ is a finite subgroup of $E$ acting on the factor $E$ by translation, and via an injective representation $G \inj F \rtimes \Aut(F)$ on  the factor $F$. There is a short list of possible groups $G$, and in most cases $F$ must have complex multiplication by either $\bZ[i]$, or the Eisenstein integers $\bZ[\zeta_3]$. The projections to the factors are equivariant with respect to the $G$-action and so $X$ sits in the two fibrations
\[
\xymatrix@M+1ex{
& \ar[dl] A \ar[d]^\pi \ar[dr] & \\
E \ar[d] & \ar[dl]_h X \ar[dr]^{f} &  F \ar[d] \\
E' = E/G& & F/G = \bP^1
}
\]
where $h$ is a smooth projective isotrivial fibration with fiber $F$, and $f$ is an isotrivial elliptic fibration with multiple fibers and general fiber $E$. Let $F_0$ be the reduced fiber of $f$ in the image of $0 \in F$ in $F/G$, and let $S \subseteq F$ be the $G$-orbit of $0 \in F$. The line bundle  $\cL = \dO_X(h^{-1}(0) + F_0)$
pulls back under $\pi$ to 
\[
\pi^\ast \cL = \dO_E(G) \boxtimes \dO_F(S)\ .
\]
It follows that $\cL$ is an ample line bundle on $X$, and 
\[
(\cL^2) = 2 \# S\ .
\]
The following groups, stabilizers $G_0$ of $0 \in F$, and self-intersection of $\cL$ occur, see \cite[page 199]{barth_compact_2004}, here $\mu_n$ denotes the group of $n$-th roots of unity, and $\zeta_3$ is a cubic root of unity.

\[
{\renewcommand{\arraystretch}{1.2}
\setlength{\arraycolsep}{1em} 
\begin{array}{ccclcc}
\toprule
G &  \#G & F & \text{action of generators} & G_0  & (\cL^2)  \\ \midrule
\mu_2 & 2 & \text{arbitrary} &       (x, y) \mapsto (x+\alpha, -y)  &  G  &  2  \\ \midrule
\mu_3 & 3 & \text{CM by } \bZ[\zeta_3] &    (x, y) \mapsto (x+\alpha, \zeta_3y)    &  G  &  2  \\ \midrule
\mu_4 & 4 & \text{CM by } \bZ[i] &      (x, y) \mapsto (x+\alpha, iy)   &  G  &  2  \\ \midrule
\mu_6 & 6&  \text{CM by } \bZ[\zeta_3] &     (x, y) \mapsto (x+\alpha, -\zeta_3y)    &  G  &  2  \\ \midrule
\mu_2 \times \bZ/2\bZ & 4 & \text{arbitrary} &     (x, y) \mapsto (x+\alpha, -y)  &  \mu_2  & 4  \\
& &  & (x, y) \mapsto (x+\beta, y+\gamma)  & &  \\ \midrule
\mu_4 \times \bZ/2\bZ & 8&  \text{CM by } \bZ[i] &      (x, y) \mapsto (x+\alpha, iy)   &   \mu_4  &   4 \\ 
& & & (x, y) \mapsto (x+\beta, y+\gamma) & &  \\ \midrule
\mu_3 \times \bZ/3\bZ & 9 & \text{CM by } \bZ[\zeta_3] &     (x, y) \mapsto (x+\alpha, \zeta_3y)  & \mu_3   &  6   \\  
& & & (x, y) \mapsto (x+\beta, y+\gamma) & &   \\ \bottomrule
\end{array}
}
\]
In the formula describing the action, the elements $\alpha, \beta$ (resp.\ $\gamma$) always denote a torsion point of suitable order of $E$ (resp.\ of $F$).

Before we can compute the convex Fujita numbers we need two lemmas for the case where $G$ is of order $9$. 

\begin{lem}
\label{lem:G9 part1}
Let $G = E[3] \subseteq E$ be the kernel of the map $[3]$, multiplication by $3$. Then the descent spectral sequence for the $G$-cover $[3]: E \to E$ for coefficients $\Gm$ yields an exact sequence
\[
0 \to \Hom(G,\bC^\times) \to \Pic(E) \xrightarrow{[3]^\ast} \Pic(E)^G \xrightarrow{d} \rH^2(G,\bC^\times) \to 0\ .
\]
Moreover, a line bundle $\cM \in \Pic(E)$ is $G$-invariant if and only if $\deg(\cM)$ is divisible by $3$.
\end{lem}
\begin{proof}
The short exact sequence of low degree terms of the spectral sequence
\[
\rE_2^{a,b} = \rH^a(G,\rH^b(E,\Gm)) \Longrightarrow \rH^{a+b}(E,\Gm)
\]
yields
\[
0 \to \Hom(G,\bC^\times) \to \Pic(E) \xrightarrow{[3]^\ast} \Pic(E)^G \xrightarrow{d} \rH^2(G,\bC^\times) \to \rH^2(E,\Gm)\ .
\]
By Tsen's theorem and the  purity of the Brauer group we find $\rH^2(E,\Gm) = 0$. This proves the first assertion. 

A line bundle $\cM$ is $G$-invariant if and only if for all $\alpha \in E[3]$ we have 
\[
\ph_{\cM}(\alpha) = t_\alpha^\ast \cM \otimes \cM^{-1} = 0\ ,
\]
so if and only if $E[3] \subseteq \ker(\ph_\cM)$. The kernel of the polarization $\ph_\cM$ consists of all $\deg(\cM)$-torsion points, hence this holds if and only if $3 \mid \deg(\cM)$ as claimed.
\end{proof}

The second lemma resembles the first but treats the ramified case. For a finite group $G$ acting on a variety $F$ we denote the orbifold quotient by $[F/G]$.

\begin{lem}
\label{lem:G9 part2}
Let $G = E[3] \subseteq E$ be the kernel of the map $[3]$  (multiplication by $3$),  and let $G$ act on the elliptic curve $F$ as in the bottom line of the table above. Then the spectral sequence for the orbifold $G$-cover $\psi: F \to [F/G]$ for coefficients $\Gm$ yields an exact sequence
\[
0 \to \Hom(G,\bC^\times) \to \Pic([F/G]) \xrightarrow{\psi^\ast} \Pic(F)^G \xrightarrow{d} \rH^2(G,\bC^\times) \to 0\ .
\]
Moreover, the restriction of $d$ to the $3$-torsion subgroup of $\Pic(F)^G$ is surjective onto 
$\rH^2(G,\bC^\times)$.
\end{lem}
\begin{proof}
The short exact sequence of low degree terms of the spectral sequence
\[
\rE_2^{a, b} = \rH^a(G,\rH^b(F,\Gm)) \Longrightarrow \rH^{a+b}([F/G],\Gm)
\]
yields
\[
0 \to \Hom(G,\bC^\times) \to \Pic([F/G]) \xrightarrow{\psi^\ast} \Pic(F)^G \xrightarrow{d} \rH^2(G,\bC^\times) \to \rH^2([F/G],\Gm)\ .
\]
Therefore,  the first assertion follows from the second. The Schur multiplier $\rH^2(G,\bC^\times)$ equals the exterior square $\bigwedge^2 G \simeq \bZ/3\bZ$. Hence it suffices to see that $d$ restricted to the $3$-torsion subgroup of $\Pic(F)^G$ is nontrivial. 

The action of the $3$-group $G$ on the $3$-group $\Pic^0(F)[3]$ of $3$-torsion has always a non-trivial fixed part. So $\Pic(F)^G[3] = \big(\Pic^0(F)[3]\big)^G$ is non-empty. We now argue by contradiction, and assume that $d$ annihilates this $G$-invariant $3$-torsion. Then $\Pic(F)^G[3]$ lies in the image of $\Pic([F/G])$ under $\psi^\ast$. More precisely, since $\ker(\psi^\ast) =  \Hom(G,\bC^\times)$ is finite, we even find a preimage in the torsion subgroup of $\Pic([F/G])$. 

The ramified cover $\bar{\psi} : F \to F/G \simeq \bP^1$, the coarse version of $\psi$, is easily seen to be ramified in three points with ramification index $3$, therefore the orbifold fundamental group of $[F/G]$ has a presentation
\[
\pi_1([F/G]) \simeq \langle a, b, c \ | \ a^3 = b^3 = c^3 = abc = 1\rangle\ .
\]
We can therefore compute the torsion subgroup
\[
\Pic([F/G])_\tors = \Hom(\pi^\ab_1([F/G]), \bC^\times_\tors) \simeq \bZ/3\bZ \times \bZ/3\bZ,
\]
so that the natural map induced by the homomorphism $\pi_1([F/G]) \surj G$ describing the $G$-torsor $F \to [F/G]$
\[
\Hom(G,\bC^\times) \to \Hom(\pi^\ab_1([F/G]), \bC^\times_\tors)  = \Pic([F/G])_\tors
\]
is an isomorphism. Therefore $\psi^\ast$ applied to torsion classes is zero, a contradiction finishing the proof.
\end{proof}

\begin{prop}
\label{prop:FN for bielliptic}
Any bielliptic surface $X$ has an ample line bundle $\cL$ with self-intersection $(\cL^2) \leq 4$. In particular, $X$ has convex Fujita number $\conFN(X) = 2$.
\end{prop}
\begin{proof}
We use the notation introduced above. We first deduce the consequence for the convex Fujita number. Since $\cL$ is ample and $\omega_X$ is numerically trivial, Kodaira vanishing for the groups $\rH^i(X, \omega_X \otimes (\omega_X^{-1} \otimes \cL))$ and Riemann-Roch show
\[
\rh^0(X,\cL) = \chi(X,\cL) = \chi(X,\dO_X) + \frac{1}{2} (\cL^2) = \frac{1}{2} (\cL^2) \leq 2\ .
\]
Arguing as in the proof of \pref{prop:globalgeneration abelian surface}, if $\rh^0(X,\cL) \leq 2$ then the ample $\cL$ cannot be globally generated. Therefore the adjoint bundle $\cL$ of the ample line bundle $\omega_X^{-1} \otimes \cL$ has a base point. This means that $\conFN(X) \geq 2$. On the other hand $\conFN(X) \leq 2$ by  \pref{prop:FNsurfacesReider}. This proves the assertion on the convex Fujita number. 

\smallskip

It remains to establish the existence of $\cL$. For all cases  but the last line in the table above, the line bundle $\cL$  listed there satisfies the claim about the self intersection number. Hence, it remains to study the case of $G = \mu_3 \times \bZ/3\bZ$. 

As above we consider the two maps $h: X \to E/G = E$ and $f: X \to [F/G]$. The exact sequence of low degree terms of the spectral sequence
\[
\rE_2^{a, b} = \rH^a(G,\rH^b(E\times F,\Gm)) \Longrightarrow \rH^{a+b}(X,\Gm)
\]
can be compared by pulling back along $h$ and $f$ with the sum of the short exact sequences of 
\lref{lem:G9 part1} and \lref{lem:G9 part2} as follows:
\[
\xymatrix@M+1ex@C-2.1ex{
0 \ar[r] & \rH^1(G,\bC^\times)^{\oplus 2} \ar[r] \ar[d]^{\sum} & \Pic(E) \times \Pic([F/G]) \ar[d]^{h^\ast(-) \otimes f^\ast(-)}\ar[r]^(0.52){[3]^\ast \times \psi^\ast} & \Pic(E)^G \times \Pic(F)^G \ar[r]^(0.56)d \ar[d]^{- \boxtimes -}& \rH^2(G,\bC^\times)^{\oplus 2} \ar[d]^{\sum}  \\
0 \ar[r] & \rH^1(G,\bC^\times) \ar[r] & \Pic(X) \ar[r]^{\pi^\ast} & \Pic(E\times F)^G \ar[r]^d & \rH^2(G,\bC^\times) .
}
\]
Let $\cM_1 \in \Pic(E)$ be of degree $3$. By \lref{lem:G9 part1} this $\cM_1$ is $G$-invariant. By  \lref{lem:G9 part2} we can pick a $3$-torsion class $\alpha$ in $\Pic(F)^G$ such that
\[
d(\cM_1 \boxtimes \alpha) = d(\cM_1) + d(\alpha) = 0
\]
Recall that $S \subseteq F$ is the $G$-orbit of $0 \in F$. Thus $\dO_F(S)$ carries a $G$-equivariant structure and so $d(\dO_F(S)) = 0$. It follows that with $\cM_2 = \dO_F(S) \otimes \alpha$ the line bundle $\cM_1 \boxtimes \cM_2$ descends to $X$: we have $\cL \in \Pic(X)$ with 
\[
\pi^\ast \cL = \cM_1 \boxtimes \cM_2\ .
\]
Clearly $\cL$ is ample, because $\pi^\ast \cL$ is ample on $E \times F$. Since both $\cM_i$ have degree $3$ we find
\[
(\cL^2) = \frac{2}{\#G}  \deg(\cM_1) \cdot \deg(\cM_2) = 2\ .
\]
This concludes the proof. 
\end{proof}

\subsection{K3 surfaces}

The improvement on Fujita's freeness conjecture for K3 surfaces is a classical result by Saint-Donat \cite[Theorem 8.3]{saint-donat_projective_1974}. It says  that $\cL^{\otimes 2}$ is globally generated for an ample line bundle $\cL$. Here we prove a precise criterion for the convex Fujita number  
(compare  Saint-Donat \cite[Proposition 8.1]{saint-donat_projective_1974} and 
Mayer \cite[Propositions 5+6]{mayer_families_1972} for the equivalence of \ref{propitem:K3ample not gg}  and \ref{propitem:K3ample not gg geometry} below).

\begin{prop}
\label{prop:FN for K3}
Let $X$ be a K3 surface equipped with an ample line bundle $\cL$. Then the following are equivalent.
\begin{enumerate}[label=(\alph*),align=left,labelindent=0pt,leftmargin=*,widest = (iii)]
\item 
\label{propitem:K3ample not gg} 
$\cL$ is not globally generated.
\item
\label{propitem:K3ample not gg geometry} 
$X$ admits an elliptic fibration $\ph: X \to \bP^1$ with general fiber $E$ and image $S$ of a section such that all fibers are irreducible and reduced and $\cL \simeq \dO_X(mE + S)$ for some $m \geq 3$.
\end{enumerate}
Both values $0$ and $2$ of the convex Fujita number occur:
\begin{enumerate}[align=left,labelindent=0pt,leftmargin=*]
\item 
\label{propitem:K3 FN0}
If $X$ does not admit an elliptic fibration, e.g.\ if $X$ has Picard number $1$ (e.g.\ if $X$ is a very general K3 surface) then $\conFN(X)$ equals $0$.
\item 
\label{propitem:K3 FN1}
There are K3 surfaces $X$ that admit the geometry of \ref{propitem:K3ample not gg geometry}. For these 
 $\conFN(X)$ equals $2$.
\end{enumerate}
\end{prop}
\begin{proof}
Let $\cL$ be an ample line bundle on a K3 surface $X$ that is not globally generated. It follows from \cite[II, Corollary 3.15]{huybrechts_lectures_2016} and its proof that $\cL = \dO_X(mE + S)$ for a smooth elliptic curve $E$ and a curve $S \simeq \bP^1$, and some $m \geq 2$, such that $S$ is the base locus of $\cL$ and $mE$ is the mobile part. By adjunction we have $E^2 = 0$ and $S^2 = -2$. In particular, the effective divisor $E$ is nef. By \cite[3.8]{reid_chapters_1997} (see also \cite[II, Prop 3.10]{huybrechts_lectures_2016}), the line bundle $\dO_X(E)$ is globally generated. The proof of \cite[II, Prop 3.10]{huybrechts_lectures_2016} shows that the regular map
\[
\ph: X \to \bP(\rH^0(X,\dO_X(E))
\]
is a fibration over a curve, necessarily a projective line, as $X$ is simply connected. Let $f: X \to \bP^1$ be its Stein factorization. Since $\dO_X(E) = \ph^\ast \dO(1) = f^\ast \dO(d)$ for some $d \geq 1$, it follows that $E$ is linearly equivalent to a multiple of a general fiber of $f$. We may replace $E$ by a general fiber of $f$ without loss of generality, so $d$ equals $1$. Now by Kodaira vanishing and Riemann--Roch we have
\[
m(E \bullet S) + 1 = 2 + \frac{1}{2}(\cL^2)  = \chi(X,\cL) = \rh^0(X,\cL)\ .
\]
Since $mE$ is the mobile part of $\cL$, this equals, by the projection formula applied to $f$,
\[
\rh^0(X,mE) = \rh^0(\bP^1,\dO(m)) = m+1\ .
\]
We deduce $E \bullet S = 1$ (as observed by Ulrike Rie\ss,  see the footnote in 
\cite[page 31]{huybrechts_lectures_2016}). This means that $f|_S : S \to \bP^1$ is the inverse of a section of $f$.

Now $\cL$ intersects any fiber in a point: $\cL \bullet E = S \bullet E = 1$. Being ample, $\cL$ intesects all components of fibers, hence all fibers are irreducible and reduced. From $0 < \cL \bullet S = m - 2$, we deduce $m \geq 3$. This establishes the geometry claimed in \ref{propitem:K3ample not gg geometry}.

\smallskip

For the converse direction\footnote{Note that $\ph: X \to \bP^1$ is a pseudosplit irreducible fibration in the sense of \dref{defi:special fibration} and that $(S^2) = -2$, so that \pref{prop:specialfibrationFN} almost proves what we need.}
we need to show that the given $\cL = \dO_X(mE+S)$ is in fact ample but not globally generated. By computing $\cL \bullet E = 1$, and $\cL \bullet S = m-2$ and $(\cL^2) = 2m-2$  we deduce that $\cL$ is ample by the Nakai--Moishezon criterion.  

It remains to establish that $S$ is the fixed part of $\cL$. The section $S$ meets each fiber in a smooth point of the fiber. Being irreducible and reduced and of arithmetic genus $0$, the fiber is either an elliptic curve or a rational line with a double point or a cusp. In all cases it follows that the restriction of $\dO_X(S)$ to the fiber has trivial $\rH^1$ and $\rh^0$ equal to $1$.  By cohomology and base change this means that $\dO_{\bP^1} \simeq f_\ast \dO_X(S)$. The projection formula yields $f_\ast \cL = f_\ast \dO_X(S) \otimes \dO(m) = \dO(m) = f_\ast \dO_X(mE)$, and consequently the canonical map 
\[
\rH^0\big(X,\dO_X(mE)\big) \to \rH^0(X,\cL)
\]
is an isomorphism. This shows that $\cL$ is not globally generated, as $S$ is in the (in fact agrees with the) base locus.

\smallskip

Assertion \eqref{propitem:K3 FN0} follows obviously from the proven equivalence because the geometry established in  \ref{propitem:K3ample not gg geometry} requires an elliptic fibration, which requires Picard rank at least $2$. This is false for  a very general K3 surface. 

For the claim \eqref{propitem:K3 FN1} we need to construct a suitable elliptically fibered K3 surface with a section. Such a K3 surface has a Weierstra\ss\ form 
\[
y^2 = x^3 + A(t) x + B(t)
\]
with $A(t)$ of degree $8$ and $B(t)$ of degree $12$. The type of singular fibers can be read off from the discriminant $\Delta(t) = -16(4A^3+27B^2)$ and the $j$-function. The generic case will be that $\Delta(t)$ has only simple roots, and these will exactly be the elliptically fibered K3 surfaces that satisfy \ref{propitem:K3ample not gg geometry} with nodal rational curves as singular fibers. The existence of such elliptic fibrations follows more concretely from \cite[Lemma 2.4]{miranda_configurations_1989}.
\end{proof}

\subsection{Enriques surfaces}
We refer to \cite{cossec_enriques_1989} as a general source for Enriques surfaces.
Any Enriques surface $X$ admits a genus one fibration $f: X \to \bP^1$, see \cite[Thm.~2.1]{lang_enriques_1983} \cite[Thm.~5.7.1]{cossec_enriques_1989} or  \cite[Thm.~17.5]{barth_compact_2004},
 and any genus one fibration has exactly two multiple fibres, see \cite[Thm.~5.7.2]{cossec_enriques_1989} or \cite[Lemma~17.1]{barth_compact_2004}, namely half fibres $F$ and $F'$ with $F' \sim F + K_X$, and the fibration being the pencil $|2F| = |2F'|$. The divisors $F$ that give rise to half fibres in genus one fibrations are characterized as nef effective classes that are primitive and isotrivial in the N\'eron-Severi lattice $\Num(X) = \NS(X)/\tors$ with respect to the intersection pairing. A useful function on the classes of big and nef divisors $D$ is 
\[
\Phi(D) = \min \{D \bullet F \ ; \ F \text{ is a half fibre of a genus one fibration on } X\}.
\]
It is known that $D$ has base points if and only if $\Phi(D) = 1$, 
combine  \cite[Thm.~8.3.1]{cossec_projective_1983}\footnote{Note that Theorem~8.3.1 is misprinted as Theorem~3.3.1 in \cite{cossec_projective_1983}.} with \cite[Thm.~4.1]{cossec_projective_1983} 
or the case $k=0$ of \cite[Thm.~1.2]{knutsen_k-th-order_2001}.

An irreducible divisor $B \subseteq X$ is a \textbf{bisection} of a genus one fibration $|2F|$ on $X$ if $B \bullet F = 1$. It is a theorem thanks to the 
unimodularity\footnote{Unimodularity follows in characteristic $0$ from Poincar\'e duality and $\Num(X) = \rH^2(X,\bZ)/\tors$, but it also holds in characteristic $p$ by \cite[Cor~7.3.7]{illusie_complexe_1979}.}
 of $\Num(X)$ that every genus one fibration admits a bisection, 
 see the proof of \cite[Thm.~2.2]{lang_enriques_1983} or \cite[Prop.~17.6]{barth_compact_2004}. 

\begin{prop}
\label{prop:FN Enriques}
Let $X$ be an Enriques surface. Then $1 \leq \conFN(X) \leq 2$ and the following are equivalent:
\begin{enumerate}[label=(\alph*),align=left,labelindent=0pt,leftmargin=*,widest = (a)]
	\item
	\label{propitem:FN Enriques1}
	  $\conFN(X) = 2$, 
	 \item
	\label{propitem:FN Enriques2}
	 There is an ample line bundle $\cL$ that is not globally generated. 
	\item
	\label{propitem:FN Enriques3}
  There is a genus one fibration $f: X \to \bP^1$ with an ample  bisection $B$.
	\item  
	\label{propitem:FN Enriques4}
There is a genus one fibration $f: X \to \bP^1$ with a bisection $B$ such that $B$ meets every component of a fibre of $f$. 
\end{enumerate}
\end{prop}
\begin{proof}
We assume \ref{propitem:FN Enriques1}: 
If $\conFN(X) = 2$, then there is an ample line bundle $\cL'$ such that $\cL = \omega_X \otimes \cL'$ is not globally generated. Since $K_X$ is numerically trivial, the line bundle $\cL$ is also ample. This shows \ref{propitem:FN Enriques2}.

Now we assume \ref{propitem:FN Enriques2}: as recalled above, this means $\Phi(\cL) = 1$ and so there is a genus one fibration $f: X \to \bP^1$ with half fibre $F$ and $\cL \bullet F = 1$. By Riemann-Roch and Kodaira vanishing we have
\[
\rh^0(X,\cL) = \chi(X,\cL) = 1 + \frac{1}{2}(\cL^2) \geq 2.
\]
In particular, $\cL$ is effective. By \cite[Thm.~8.3.1]{cossec_projective_1983}  there is an integral divisor $B$ such that $\cL \simeq \dO_X(B)$. This $B$ cannot be a fibre component of $f$, hence must be finite over $\bP^1$ and since $B \bullet F = \cL \bullet F = 1$ we find that $B$ is an ample bisection. This shows \ref{propitem:FN Enriques3}.

Now we assume \ref{propitem:FN Enriques3}: Since an ample divisor meets any irreducible curve, the ample bisection in particular meets all fibre components. This shows \ref{propitem:FN Enriques4}.

Now we assume \ref{propitem:FN Enriques4}: Let $B$ be a bisection that meets every fibre component of the genus one fibration $|2F|$. Let $D$ be an effective divisor on $\bP^1$ of degree $\deg(D) > - \frac{1}{2}(B^2)$. We set $\cL = \dO_X(B + f^\ast D)$. Then $\cL$ is ample by the Nakai-Moishezon criterion. Indeed, horizontal curves intersect positively with $f^\ast D$ while, by assumption, fibre components intersect positively with $B$. The self intersection is
\[
(\cL^2) = (B^2) + \deg(D) (B \bullet 2F) = (B^2) + 2 \deg(D)  > 0.
\]
The line bundle $\omega_X \otimes \cL$ is also ample and has $\Phi(\omega_X \otimes \cL) = 1$ as witnessed by the genus one fibration $f$. Thus $\omega_X \otimes \cL$ is not globally generated and $\conFN(X) > 1$. This shows  \ref{propitem:FN Enriques1} completing the proof.
\end{proof}

Recall that an Enriques surface is called \textbf{nodal} if $X$ contains a smooth rational curve, and it is called \textbf{unnodal} if no such smooth rational curves exist. The moduli space of nodal Enriques surfaces is a divisor in the $10$-dimensional moduli space of all Enriques surfaces.

\begin{prop}
\label{prop:FN Enriques unnodal}
An unnodal Enriques surface $X$ has convex Fujita number $\conFN(X) = 2$.
\end{prop}
\begin{proof}
Any reducible fibre of a genus one fibration consists of smooth rational curves on $X$. If $X$ is unnodal, then all genus one fibrations have only irreducible fibres\footnote{The genus one fibrations in the unnodal case are all pseudosplit irreducible, hence  \cref{cor:specialfibrationFN}  applies.}  and so every bisection meets all components of fibres. The result thus follows from \pref{prop:FN Enriques}.
\end{proof}

\begin{cor}
\label{cor:FN Enriques}
Let $X$ be an Enriques surface.
\begin{enumerate}[label=(\arabic*),align=left,labelindent=0pt,leftmargin=*,widest = (3)]
	\item
	If $X$ admits a genus one fibration with all fibres irreducible, then $\conFN(X)$ equals $2$.
	 \item
	If all genus one fibrations of $X$ have a fibre with at least three irreducible components, then $\conFN(X)$ equals $1$.
\end{enumerate}
\end{cor}
\begin{proof}
Immediately from \pref{prop:FN Enriques}. In case (1) property \ref{propitem:FN Enriques4} is satisfied for some genus one fibration, and in case (2) property \ref{propitem:FN Enriques4} is never satisfied. 
\end{proof}

In \cite{brandhorst_automorphism_2022} the authors study possible ADE-types of $(-2)$-curves on Enriques surfaces $X$. A first interesting fact is that up to $\Aut(X)$ there are only finitely many genus one fibrations and that representatives of the $\Aut(X)$-orbits can be computed by lattice theoretic algorithms when asked for $(\tau,\bar \tau)$-generic Enriques surfaces. For the definition of the latter we refer to \cite{brandhorst_automorphism_2022}. 
The paper furthermore contains a table 
\cite[\S6.5]{brandhorst_automorphism_2022}\footnote{The list in \textit{loc.\ cit.} is truncated and has some artificial page breaks that are hard to follow. The complete list can be found in \href{http://www.math.sci.hiroshima-u.ac.jp/shimada/K3andEnriques.html}{http://www.math.sci.hiroshima-u.ac.jp/shimada/K3andEnriques.html}, published also in zenodo, 
\href{https://doi.org/10.5281/zenodo.4327019}{https://doi.org/10.5281/zenodo.4327019}.
}
listing the ADE-types of singular non-multiple fibres and of the two half fibres in any possible genus one fibrations of a $(\tau,\bar \tau)$-generic Enriques surfaces. Upon inspecting the list we arrive at the following proposition.

\begin{prop}
\label{prop:FN1 Enriques}
There are Enriques surfaces $X$ with convex Fujita number $\conFN(X) = 1$. 
\end{prop}
\begin{proof}
Of the list as contained in  \cite[\S6.5]{brandhorst_automorphism_2022} we pick two examples of numbers with  generic ADE-types $(\tau,\bar \tau)$ admit only genus one fibrations with some fibre having at least $3$ irreducible components. According to \cref{cor:FN Enriques}, these Enriques surfaces have $\conFN(X) = 1$.
\[
{\renewcommand{\arraystretch}{1.2}
\setlength{\arraycolsep}{1em} 
\begin{array}{clcl}
\toprule
\text{No.} &  (\tau, \bar \tau) & \text{singular non-multiple fibres} &  \text{singular half fibres}    \\ \midrule
24  & (D_5,D_5) &  \text{none} &  A_3    \\
& & \text{none} & A_4 \\
& & A_3 + 2A_1 & \text{none} \\
& & A_4 & \text{none} \\
& & D_4 & \text{none} \\
& & D_5  & \text{none} \\ \midrule
47 & (E_6,E_6) &  \text{none} &  A_4    \\
& & A_5 + A_1 & \text{none} \\
& & D_5 & \text{none} \\
& & E_6  & \text{none} 
    \\ \bottomrule
\end{array}
}
\]
For $49$ of the $155$ Enriques surfaces on the list, the argument above relying on \cref{cor:FN Enriques} decides that $\conFN(X)$ equals $1$.
\end{proof}

\section{Elliptic fibrations of Kodaira dimension \texorpdfstring{$1$}{1}}
\label{sec:FN Kodaira 1}

\subsection{Preliminaries on elliptic fibrations of Kodaira dimension $1$}
\label{sec:genus one}

Smooth projective surfaces of Kodaira dimension $1$ all admit the structure of an elliptic fibration $f: X \to C$. Here $C$ is a smooth projective curve, and $f$ is a fibration such that the general fibre has arithmetic genus $1$. 
For $P \in C$ we denote by $m_P$ the multiplicity of the fibre $f^{-1}(P)$ and by $F_P$ the divisor supported in the fibre such that $f^{-1}(P) = m_P F_P$. 
Kodaira's canonical bundle formula \cite[V. Theorem 12.1]{barth_compact_2004} reads
\begin{equation}
\label{eq:Kodaira canonical bundle formula}
\omega_X = f^\ast  f_\ast \omega_X \otimes \dO_X(\sum_P (m_P-1)F_P)\ , 
\end{equation}
with $f_\ast \omega_X = \cHom(\RR^1 f_\ast \dO_X, \omega_C)$.

We start with the following result. 

\begin{prop}
Let $f: X \to C$ be an elliptic fibration. 
\begin{enumerate}[label=(\arabic*),align=left,labelindent=0pt,leftmargin=*,widest = (3)]
\item
\label{propitem:baselocusmultiplefibres}
If $f$ has multiple fibres, then $\omega_X$ is not globally generated. The divisors $F_P$ where $f^{-1}(P) = m_P F_P$ is a multiple fibre are contained in the base locus. 
\item 
\label{propitem:baselocusnomultiplefibres}
Let us assume that $f$ has no multiple fibre. Then the following holds. 
\begin{enumerate}[label=(\alph*),align=left,labelindent=0pt,leftmargin=*,widest = (a)]
\item 
If $\chi(X,\dO_X) \geq 2$, then $\omega_X$ is globally generated.
\item 
If $\chi(X,\dO_X) = 1$ and $C$ has genus $g \leq 1$, then $\omega_X$ is not globally generated.
\end{enumerate}
\end{enumerate}
In cases \ref{propitem:baselocusmultiplefibres} and \ref{propitem:baselocusnomultiplefibres}(b) we have $\conFN(X) \geq 1$, while in case \ref{propitem:baselocusnomultiplefibres}(a) we have $\conFN(X) \not= 1$.
\end{prop}
\begin{proof}
The projection formula applied to \eqref{eq:Kodaira canonical bundle formula} shows that all sections of $\omega_X$ have poles of order $m_P-1$ along $F_P$. This shows \ref{propitem:baselocusmultiplefibres}.

We now show \ref{propitem:baselocusnomultiplefibres} and assume to this end that $f$ has no multiple fibres. Then 
\[
f_\ast \omega_X = \omega_C \otimes \cHom(\RR^1 f_\ast \dO_X, \dO_C) 
\]
where the line bundle $\cL = \cHom(\RR^1 f_\ast \dO_X, \dO_C)$ has degree $\chi(X,\dO_X) \geq 0$ by \cite[V. Corollary 12.3]{barth_compact_2004} and \cite[III. Theorem 18.2]{barth_compact_2004}. Since $f_\ast \dO_X = \dO_C$, the map 
\[
\rH^0(C,\omega_C \otimes \cL) \to \omega_C \otimes \cL
\]
controlling generation by global sections for $\omega_C \otimes \cL$ pulls back via $f^\ast$ to the map controlling generation by global sections for $\omega_X$
\[
\rH^0(X,\omega_X) \otimes \dO_X = f^\ast\big(\rH^0(C,f_\ast \omega_X) \otimes \dO_C \big) \to f^\ast (\omega_C \otimes \cL) = \omega_X\ .
\]
As $f$ is faithfully flat, global generation for $\omega_X$ is equivalent to global generation for $\omega_C \otimes \cL$.

Since $C$ as a curve has convex Fujita number $\conFN(C) = 2$, the line bundle $\omega_X \otimes \cL$ is globally generated if $\deg(\cL) = \chi(X,\dO_X) \geq 2$. If $\chi(X,\dO_X) = 1$, then by Riemann--Roch and Serre duality the line bundle $\omega_X \otimes \cL$ is not globally generated if and only if $\cL$ is effective. If the genus of $C$ is $\leq 1$, then any line bundle of degree $1$ is effective.
\end{proof}

\begin{rmk}
As shown in \cite[V. Proposition 12.5]{barth_compact_2004}, an  elliptic fibration has indeed Kodaira dimension $1$ if  the following $\delta_f$ is positive:
\begin{equation}
\label{eq:numerical criterion kodaira 1}
\delta_f = \chi(X,\dO_X) - 2 \chi(C,\dO_C) + \sum_P (1- \frac{1}{m_P}) > 0 \ ,
\end{equation}
and this follows for example if the genus of $C$ is at least $2$.
\end{rmk}

\begin{prop}
\label{prop:elliptic fibration FN2}
Let $X \to C$ be an elliptic fibration with a section, such that
\begin{enumerate}[label=(\roman*),align=left,labelindent=0pt,leftmargin=*,widest = (iii)]
\item all fibers are irreducible and reduced, and 
\item $\chi(X,\dO_X)$ is even.
\end{enumerate}
Then $X$ has convex Fujita number $\conFN(X) = 2$. Moreover, elliptic fibrations of Kodaira dimension $1$ as above exist. 
\end{prop}
\begin{proof}
An elliptic fibration with a section does not have multiple fibres. The formula \eqref{eq:Kodaira canonical bundle formula} thus simplifies to 
\[
\omega_X = f^\ast \cHom(\RR^1 f_\ast \dO_X, \omega_C) \ .
\]
Since $\chi(X,\dO_X)$ is even, the canonical bundle $\omega_X$ is the pull back of a line bundle of even degree. It follows that $X$ has even intersection form. Hence $\conFN(X) \leq 2$ by \pref{prop:FNsurfacesReider}.

The converse estimate $\conFN(X) \geq 2$ follows from \cref{cor:specialfibrationFN} as the projection $X \to C$ is a pseudosplit irreducible fibration in the sense of \dref{defi:special fibration}.

It remains to construct an elliptic fibration as in the statement of the proposition having Kodaira dimension $1$. In \pref{prop:FN for K3} we cite \cite[Lemma 2.4]{miranda_configurations_1989} for the existence of a K3 surface $f_0: X_0 \to C_0=\bP^1$ with an elliptic fibration with a section and all fibres irreducible and reduced. We choose a ramified cover $C \to C_0$ with genus of $C$ at least $2$ and ramified only above points of $C_0$ where the map $f_0$ is smooth. Then $X = X_0 \times_{C_0} C$ is again smooth with $f: X \to C$ having a section and all fibres irreducible and reduced. By 
\eqref{eq:numerical criterion kodaira 1} this $X$ has Kodaira dimension $1$, because $\chi(X,\dO_X) \geq 0$.  Finally, we need to ensure that $X$ has even Euler characteristic. This will be achieved by a further \'etale base change $X'  = X \times_C C'$ with an \'etale double cover $C' \to C$. The base change is still smooth, and the geometric assumptions are preserved in the genus one fibration $f' : X' \to C'$. But in addition $\chi(X',\dO_{X'}) = \deg(C'/C) \chi(X,\dO_X)$ is necessarily even.
\end{proof}

\subsection{An isotrivial elliptic fibration}
\label{sec:isotrivial elliptic fibration Kodaira 1 FN3}

Next, we describe an isotrivial elliptic fibration with convex Fujita number $3$. A considerable amount of the analysis is stimulated by reading \cite{serrano_elliptic_1992}.

\begin{prop}
\label{prop:elliptic fibration FN3}
There is a minimal smooth projective surface of Kodaira dimension $1$ and convex Fujita number $3$. 

More precisely, let $E$ be an elliptic curve, and let $C \to \bP^1$ be a branched cover with Galois group $G = E[2]$ and $C$ of genus $2$. We assume that $E$ is not an isogeny factor of $\Pic^0(C)$. Then $X = E \times C/G$ with $G$ acting factorwise (by translation on $E$) is an isotrivial elliptic fibration that  is a minimal smooth projective surface of Kodaira dimension $1$ and convex Fujita number $\conFN(X) = 3$. 
\end{prop}

The proof spans over several lemmas and computations.

\begin{lem}
A cover $C \to \bP^1$ as in \pref{prop:elliptic fibration FN3} exists. 
\end{lem}
\begin{proof}
By the Riemann--Hurwitz formula there will be $5$ branch points in $\bP^1$, hence we must show that there is an $E[2]$-quotient of
\[
\pi_1(\bP^1 - \{P_1,P_2,P_3,P_4,P_5\}) = \langle c_1, c_2, c_3, c_4, c_5 \ | \ \prod_i c_i = 1\rangle
\]
such that all $c_i$ map to nontrivial elements. This is a quick computation  in the $2$-dimensional $\bF_2$-vector space $E[2]$.
\end{proof}

By choosing $C$ first, we may then choose $E$ which does not occur as an isogeny factor of $\Pic^0(C)$, for example because there are at most countably many such isogeny factors.

We consider the following diagram of maps.
\[
\xymatrix@M+1ex{
Y = E \times C \ar[r]^\pi & X = Y/G \ar[r]^h \ar[d]^f & E/G \simeq E \\
& C/G = \bP^1 & 
}
\]
Here $f$ and $h$ are the maps induced by the coordinate projections, and $E \simeq E/G$ is induced by multiplication by $2$. As $E \to E/G$ is \'etale, the map $h$ is a smooth projective isotrivial fibration with general fiber $H$ isomorphic to $C$. The general fiber $F$ of $f$ is isomorphic to $E$. The multiple fibers of $f$ are the fibers in the branch points $P \in C/G$. Let $F_P$ denote the reduced fiber of $f$ in $P$. Then the multiplicity of $F_P$ equals the ramification index $e_P$. So numerically $F \equiv e_P F_P$.  

\begin{lem}
\label{lem:NSbasis}
The classes of $F$ and $H$ form a basis of $\NS(X) \otimes \bQ$  and 
\[
\pi^\ast : \NS(X) \otimes \bQ \to \NS(Y) \otimes \bQ
\]
is an isomorphism.
\end{lem}
\begin{proof}
The map $\pi^*$ is injective since $\pi$ is finite.  By assumption $E$ and $\Pic^0(C)$ are disjoint, hence there are no correspondences and $\NS(E \times C) = \NS(E) \times \NS(C)$ which is spanned by the respective fiber classes. Since these classes are rational multiples of $F$ and $H$, the lemma follows.
\end{proof}

We next compute the class of $\omega_X$. Computing in $\NS(Y) \otimes \bQ$ yields
\[
\pi^\ast \omega_X = \omega_Y = \omega_E \boxtimes \omega_C = \frac{2g-2}{\#G} \pi^* F 
\]
and therefore $\omega_X \equiv \frac{1}{2} F$.

\begin{lem}[cf. {\cite[Proposition 1.4]{serrano_elliptic_1992}}]
\label{lem:primitive}
The class of $\omega_X$ is primitive in the lattice $\NS(X)/\tors$.
\end{lem}
\begin{proof}
\textit{Step 1.} 
Claim: The class $\lambda F$ with $2 \lambda \notin \bZ$ is not effective. 

Assuming the contrary, there would be an effective divisor $D \equiv \lambda F$. Then $D \bullet F = 0$ shows that $D$ is a vertical divisor with respect to the elliptic fibration $f: X \to \bP^1= C/G$. Since all fibers are irreducible, we find that numerically
\[
D \in \langle F_P \ ; \ P \in C/G\rangle = \frac{1}{2}\bZ \cdot F,
\]
a contradiction.

\smallskip
\textit{Step 2.}
Claim: If the class $\lambda F$ with $2\lambda \notin  \bZ$ is the class of a line bundle $\cL$, then $\cL$ has trivial cohomology $\rH^i(X,\cL) = 0$  in all degrees.

Indeed, with $\cL$ also $\omega_X \otimes \cL^{-1} \equiv (1/2 - \lambda)F$ satisfies the assumption of Step 1. Therefore $\rH^0(X,\cL) = 0$ by Step 1,  and $\rh^2(X,\cL) = \rh^0(X,\omega_X \otimes \cL^{-1}) = 0$ by Serre duality and Step 1. It remains to compute $\rH^1(X,\cL)$ via Riemann--Roch:
\[
- \rh^1(X,\cL) = \chi(X,\cL) = \chi(X,\dO_X) + \frac{1}{2} \cL \bullet (\cL - \frac{1}{2}F) = \chi(X,\dO_X).
\] 
Since $\pi: Y \to X$ is finite \'etale, the  holomorphic Euler  charactersitic multiplies by $\deg(\pi)$, and it vanishes due to the factor $E$ with $\chi(E,\dO_E) = 0$.

\smallskip
\textit{Step 3.} Claim: a line bundle $\cL$ as in Step 2 has trivial cohomology $\rH^i(F,\cL|_F) = 0$ when restricted to $F$.

We consider the exact sequence
\[
0 \to \cL(-F) \to \cL \to \cL|_F \to 0
\]
and observe that $\cL(-F)$ also satisfies the assumption of Step 2. The claim thus follows from the long exact sequence of cohomology and the vanishing computed in Step 2.

\smallskip
\textit{Step 4.} We now modify $\cL$ by a numerically trivial line bundle. The conclusion of Step $3$ remains valid as the reasoning there only depends on the numerical equivalence class of $\cL$. The composite 
\[
h|_F \colon  F \to  E
\]
is isomorphic to  multiplication by $2$ on $E$. The pull back map 
\[
\Pic^0(E) \xrightarrow{h^\ast} \Pic^0(X) \xrightarrow{-|_F} \Pic^0(F) 
\]
is therefore surjective. This means that we can prescribe $\cL|_F$ in its numerical class at will. In particular, for a specific choice of $\cL$ we obtain $\cL|_F \simeq \cO_F$. The trivial line bundle has nontrivial cohomology in all degrees $i = 0, 1$, contradicting Step 3. This completes the proof of the Lemma.
\end{proof}

\begin{lem}[cf.  {\cite[page 194]{serrano_elliptic_1992}}]
\label{lem:hodgetype}
The cohomology $\rH^2(X)$ is pure of Hodge type $(1,1)$. 
\end{lem}
\begin{proof}
Since we already know that $\NS(X) \otimes \bQ$ is spanned by the basis $F,H$, the lemma follows from the computation of the second Betti number as $b_2(X) = 2$.

Since $G = E[2]$ acts by translation on $E$, this group action is the restriction of an action by a connected algebraic group $E$. It follows by homotopy invariance of cohomology that this group action is trivial on cohomology. We now compute using the K\"unneth formula
\begin{align*}
\rH^*(X,\bQ) & = \rH^*(Y,\bQ)^G = \Big( \rH^*(E,\bQ) \otimes \rH^*(C,\bQ)\Big)^G \\
& = \rH^*(E,\bQ) \otimes \rH^*(C,\bQ)^G = \rH^*(E,\bQ) \otimes \rH^*(C/G,\bQ).
\end{align*}
Since $C/G \simeq \bP^1$, we can read off $b_2(X) = 2$. 
\end{proof}

\begin{lem}
\label{lem:unimopdular}
The intersection pairing on $\NS(X)/\tors$ is unimodular.
\end{lem}
\begin{proof}
The Lefschetz-$(1,1)$-Theorem and \lref{lem:hodgetype} imply that $\NS(X)/\tors = \rH^2(X,\bZ)$. The claim now follows from Poincar\'e duality.
\end{proof}

\begin{proof}[Proof of Proposition~\ref{prop:elliptic fibration FN3}]
The class $\omega_X = \frac{1}{2} F$ is primitive by \lref{lem:primitive}. Since the intersection product is unimodular by \lref{lem:unimopdular}, there exists a line bundle $\cL$ with $\frac{1}{2} F \bullet \cL = 1$. By \lref{lem:NSbasis} there are rational numbers $\alpha, \beta$ with $\cL \equiv \alpha F + \beta H$. Since 
\[
F \bullet H  =\frac{1}{\deg(\pi)} \pi^* F \bullet \pi^* H =  \#G = 4
\]
we must have $\beta = \frac{1}{2}$. By the adjunction formula the following is an even integer:
\[
\cL \bullet (\cL \otimes \omega_X) = (\alpha F +  \frac{1}{2}H) \big((\frac{1}{2}+\alpha)F + \frac{1}{2}H\big) =  4\alpha  + 1.
\]
Therefore $(\alpha - \frac{1}{4})F$ lies in $\frac{1}{2}F\bZ$,  and subtracting this class from $\cL$ and renaming yields
\[
\cL \equiv \frac{1}{4} F + \frac{1}{2}H  \qquad \text{ with } \quad \omega_X \bullet \cL = 1.
\]
The classes $F$ and $H$ are fibers of maps, and since $\NS(X) \otimes \bQ$ is of dimension $2$, they span the nef cone. It follows that $\cL$ is an ample line bundle. 

The proof of \pref{prop:elliptic fibration FN3} will be complete when we find a nonempty base locus for 
\[
\omega_X \otimes \cL^{\otimes 2} \equiv \frac{1}{2} F + 2(\frac{1}{4} F + \frac{1}{2}H ) = F + H.
\]
More precisely, we will shift the second factor $\cL$ by the numerically trivial class to $\cL'$ such that 
\[
\omega_X \otimes \cL \otimes \cL' \simeq \dO_X(F+H).
\]
Now we claim that $\dO_X(F+H)$ has $4$ base points in the $4$ intersection points $F \cap H$. First, since $f_\ast \dO_X = \dO_{\bP^1}$ we have by the projection formula
\[
\rh^0(X,\dO_X(F)) = \rh^0(\bP^1, \dO(1)) = 2.
\]
Next, let $S \subseteq C$ be a free $G$-orbit, and $T \subset E$ the preimage under $E \to E/G = E$ of the point of $E$	giving rise to the fiber $H$. Since the $G$-action on $E$ is fixed point free, we have 
\[
\rH^0(E,\dO_E(T)) = \rH^*(E,\dO_E(T)) = \bC[G]
\]
as a $G$-representation (Atiyah-Bott fixed point formula). Also note that for a $G$-module $V$ with $V_0$ denoting the same vector space with trivial $G$-action, we have an isomorphism as $G$-representations
\[
V \otimes \bC[G] \simeq V_0 \otimes \bC[G].
\]
Now we compute
\begin{align*}
\rH^0(X,\dO_X(F+H)) & = \rH^0\big(Y,\pi^\ast \dO_X(F+H)\big)^G 
= \rH^0\big(Y,\dO_E(T) \boxtimes \dO_C(S)\big)^G \\
& =  \big(\rH^0(E,\dO_E(T)) \otimes \rH^0(C,\dO_C(S))\big)^G =  \big(\bC[G] \otimes \rH^0(C,\dO_C(S))\big)^G \\
& = \big(\bC[G] \otimes \rH^0(C,\dO_C(S))_0\big)^G  \simeq \big(\bC[G] \big)^G \otimes \rH^0(C,\dO_C(S))_0 \\
& \simeq  \rH^0(C,\dO_C(S))_0,
\end{align*}
and Riemann-Roch for $C$ computes
\[
\rh^0(X,\dO_X(F+H)) = \rh^0(C,\dO_C(S)) = \deg(S) + 1- g = 3.
\]
Let $Z = F \cap H$ be the intersection, a reduced subscheme of degree $4$. Now we analyze the base locus of $\dO_X(F+H)$ along $H$ using the exact sequence
\[
0 \to \dO_X(F) \to \dO_X(F+H) \to \dO_H(Z) \to 0.
\]
Taking global sections yields the exact sequence
\[
0 \to \rH^0(X,\dO_X(F)) \to \rH^0(X,\dO_X(F+H)) \to \rH^0(H,\dO_H(Z)).
\]
The computation of dimensions above show that global sections of $\dO_X(F+H)$ restrict to a $1$-dimensional space of sections along $H$, which cannot generate the line bundle $\dO_X(F+H)|_H = \dO_H(Z)$ of degree $4$. So finally we conclude that $\dO_X(F+H) = \omega_X \otimes \cL \otimes \cL'$ has non-trivial base points, and this concludes the proof of  \pref{prop:elliptic fibration FN3}.
\end{proof}

\section{Surfaces of general type}
\label{sec:FN Kodaira 2}

Minimal surfaces of general type form a vast landscape, and it is therefore certainly not to be expected that we can uniformly describe the geometry that leads to the various possible convex Fujita numbers.  On the other hand, it is reasonable to expect that all numbers actually occur in the range allowed by the general result \pref{prop:FNsurfacesReider}, that is $0 \leq \conFN(X) \leq 3$. 

We will show that on minimal surfaces of general type the convex Fujita numbers $0$, $2$, and $3$ do occur, leaving the case of $\conFN(X)=1$ open for the time being. 

\begin{prop}
\label{prop:surface general type FN0 simply connected}
A very general hyperplane $X$ in $\bP^3$ of degree $d \geq 5$ is simply connected, has convex Fujita number $\conFN(X) = 0$ and is minimal of general type.
\end{prop}
\begin{proof}
This was proven in \cite[Proposition 3.2]{chen_convex_2023}.  We recall the proof here for convenience of the reader.

The hyperplane is simply connected due to the Lefschetz hyperplane theorem for the fundamental group. By Noether-Lefschetz \cite{lefschetz_certain_1921} and degree $d \geq 4$ (see also \cite[Exp XIX, Th\'eor\`eme 1.2]{deligne_groupes_1973}), the Picard group $\Pic(X)$ is generated by $\dO(1)|_X$. This is the reason to restrict to very general hyperplanes. 

The line bundle $\cL = \dO(a)|_X$ is ample if and only if $a \geq 1$, and it is globally generated if and only if $a \geq 0$. So all ample line bundles are globally generated. By adjunction and our choice of $d$, the canonical bundle $\omega_X$ equals $\dO(d-4)|_X$ and in particular is ample. It follows at once that $X$ has convex Fujita number $0$. 
\end{proof}

\begin{prop}
\label{prop:products of curves}
Let $X= C_1 \times C_2$ be a product of smooth projective curves of genus at least $2$. Then $X$ is minimal of general type with convex Fujita number $\conFN(X) = 2$.
\end{prop}
\begin{proof}
Curves have convex Fujita number $2$, see \cite[Example 1.6]{chen_convex_2023}. 
Thus we have $\conFN(C_1 \times C_2) \geq 2$ by \pref{prop:FNforProduct}. The converse inequality follows from \pref{prop:FNsurfacesReider} \eqref{propitem:reider4}, because $X$ is spin: the canonical class $\omega_X = \omega_{C_1} \boxtimes \omega_{C_2}$ is divisible by $2$.
\end{proof}

\begin{prop}
\label{prop:FermatQuinticGodeaux}
The classical Godeaux surface $X = Y/\mu_5$ that is the quotient of the Fermat quintic 
\[
Y = \{x_1^5 + x_2^5 + x_3^5 + x_4^5 = 0 \} \subseteq \bP^3
\]
by the action of $\mu_5$ given by 
\[
\zeta.[x_1:x_2:x_3:x_4] = [\zeta x_1 : \zeta^2 x_2: \zeta^3 x_3 : \zeta^4 x_4]
\] 
is minimal of general type and has convex Fujita number $3$.
\end{prop} 
\begin{proof}
Let $D_i = \{x_i = 0\} \cap Y$ be the intersection of $Y$ with the respective coordinate hyperplane. The $D_i$ are smooth, connected, and $\mu_5$-equivariant. We denote by $C_i = D_i/\mu_5$ their image in $X$. The divisor $D = D_1 + D_2$ is a divisor with normal crossing in 
\[
D_1 \cap D_2 = \{[0:0:\zeta:1] \ ; \ \zeta \in \mu_5\}.
\]
Since the quotient map $\pi: Y \to X$ is finite \'etale, It follows that $C = C_1 + C_2$ is also a normal crossing divisor and $C_1 \bullet C_2 = \frac{1}{5} (D_1 \bullet D_2) = 1$. We now apply the criterion of 
\pref{prop:NCDampleFujitaExtremeSurface}.
\end{proof}

\begin{rmk}
\label{rmk:QuinticGodeauxFN3}
The Fermat quintic used in \pref{prop:FermatQuinticGodeaux} is actually not particularly special. Miyaoka \cite[Theorem 5]{miyaoka_tricanonical_1976} and Reid \cite{reid_surfaces_1978} have constructions an $8$-dimensional moduli of numerical Godeaux surfaces with $\pi_1 = \mu_5$. For our purposes we restrict to those with ample canonical bundle. Then the universal cover is shown to be a smooth $\mu_5$-invariant quintic in $\bP^3$, where $\mu_5$ acts as in \pref{prop:FermatQuinticGodeaux}. So the quintic is a polynomial of the form
\[
F(x) = \sum_{i=1}^4 a_i x_i^5 + \sum_{i=1}^4 b_i x_i^3 x_{3i} x_{9i} + \sum_{i=1}^4 c_i (x_i x_{2i})^2 x_{4i}
\]
with indices considered modulo $5$. By rescaling the $x_i$ we may assume that $c_i = 1$ for all $i =1, \ldots, 4$. The $8$ remaining parameters $a_i$, and $b_i$ form an $\bA^{8}$, and a dense open yields a smooth quintic hyperplane that avoids the $\mu_5$-fixed points. Miyaoka shows in \cite[Theorem 5]{miyaoka_tricanonical_1976} that this space of parameters is quasi-finite over the moduli space of such surfaces obtained as quotients $X$ by $\mu_5$ acting on the hyperplane $\{F(x) = 0\}$. 

For any nontrivial torsion line bundle $0 \not=\alpha \in \pi^\ab_1(X) \simeq \Pic(X)_\tors$,  let $C_\alpha$ be the unique effective divisor in the class $K_X + \alpha$ by 
\cite[Proposition 0.4]{reid_surfaces_1978}.  Then for nontrivial $\alpha, \beta \in \pi^\ab_1(X)$ we have
\[
C_\alpha \bullet C_\beta = (K_X + \alpha) \bullet (K_X + \beta) = K_X^2 = \frac{1}{5} (\dO(1) \bullet \dO(1) \bullet \dO(5)) = 1.
\]
Again we deduce from \pref{prop:NCDampleFujitaExtremeSurface} that $\conFN(X) = 3$.
\end{rmk}

\begin{prop}
\label{prop:surface general type FN1}
Let $A$ be an abelian surface with Picard rank $1$ and symmetric principal polarization $\Theta$. Let $B$ be a general smooth divisor in the linear system $|2\Theta|$.
Let $X \to A$ be the double cover branched along $B$.

If $X$ has Picard number $1$, then $X$ is a smooth projective minimal surface of general type with ample canonical bundle $\omega_X \simeq f^\ast \dO_A(\Theta)$ and convex Fujita number $\conFN(X) = 1$.
\end{prop}
\begin{proof}
As $B$ is smooth, the double cover is a smooth projective surface $X$. Let $R \subseteq X$ be the reduced preimage of $B$. Then 
\[
\omega_X = f^\ast (\omega_A(B)) \otimes \dO_X(-R) = f^\ast \dO_A(\Theta).
\]
It follows that $\omega_X$ is ample and so $X$ is a  minimal surface of general type. The sections of $\omega_X$ are 
\[
\rH^0(X,\omega_X) = \rH^0(A, f_\ast \omega_X) = \rH^0(A, \dO_A(\Theta) \otimes f_\ast \dO_X) = \rH^0(A, \dO_A(\Theta)) \oplus \rH^0(A, \dO_A).
\]
By Riemann-Roch on $A$  and $(\Theta^2) = 2$ we find $\rh^0(X,\omega_X) = 2$. In particular, $\omega_X$ is not globally generated, as otherwise there would be a map $X \to \bP^1$ that pulls back $\dO(1)$ to an ample line bundle on $X$, contradiction. 

It remains to show that for all ample line bundles $\cL$ on $X$ the adjoint bundle $\omega_X \otimes \cL$ is globally generated. We claim that the pullback map 
\[
f^\ast \colon \Pic(A) \to \Pic(X)
\]
is an isomorphism. By \cite[Cor.~2.7]{nori_zariskis_1983}, see also \cite[Prop.~1]{kharlamov_numerically_2014} the induced map 
\[
f_\ast \colon \pi_1(X) \xrightarrow{\sim} \pi_1(A)
\]
is an isomorphism. It follows that $f^\ast : \Pic^0(A) \to \Pic^0(X)$ is an isomorphism on torsion elements, hence an isomorphism, and secondly, that the induced map $f^\ast : \NS(A) \to \NS(X)$ is an isomorphism on the torsion subgroup $\NS_\tors \simeq \Hom((\pi_1^\ab)_\tors,\bQ/\bZ)$. Therefore $\NS(X)$ is torsion free. By assumption,
 $X$ has Picard number $1$, and so it remains to show that $f^\ast \Theta$ is primitive in $\NS(X)$, since $\NS(A)$ is generated by $\Theta$. 
 
First, for any divisor $D$ on $X$ we have that $D \bullet \omega_X = f_\ast D \bullet \Theta \in (\Theta^2) \bZ$ is even, because $f_\ast D$ is numerically a multiple of $\Theta$. It follows from Riemann-Roch that the intersection form on $\NS(X)$ is even. Now, if $f^\ast \Theta \equiv nD$, then $n^2 (D^2) = 2 (\Theta^2) = 4$. As $(D^2)$ is even, we must have $n=\pm 1$. This completes the proof of the claim that $f^\ast : \Pic(A) \to \Pic(X)$ is an isomorphism. 

Any ample $\cL$ on $X$ now is of the form $\cL = f^\ast \cM$ for a necessarily ample  line bundle $\cM$ on $A$. The adjoint line bundle $\omega_X \otimes \cL$ is identified with the pullback of the line bundle $\cM \otimes \dO_A(\Theta)$ which is globally generated since $\conFN(A) \leq 2$ by the classical Lefschetz theorem. Hence also $\omega_X \otimes \cL$ is globally generated and the proof of $\conFN(X) = 1$ is complete. 
\end{proof}

\begin{rmk}
At the moment we do not know whether the assumption in \pref{prop:surface general type FN1} that $X$ has Picard number $1$ can be achieved in an example.
\end{rmk}

\section{Fujita extreme surfaces}
\label{sec:FN3}

We recall that a smooth projective surface $X$ is called Fujita extreme if $\conFN(X) = \dim\,X+1 =3$. We already observed in \pref{prop:FNsurfacesReider} that by Reider's method $\conFN(X) \leq 3$ holds for all surfaces $X$. Certainly,  $\conFN(X) = 3$ apparently requires a surface with special geometry, at the same time we will see that such surfaces are birationally cofinal over any given one. 

\subsection{Symmetric products of curves}
Recall that the symmetric square of a smooth projective curve $C$ is the smooth projective surface
\[
\Sym^2(C) = (C \times C)/S_2
\]
where $S_2$ acts by permuting the factors.

\begin{prop}
For any smooth projective curve $C$ we have $\conFN(\Sym^2(C))  = 3$.
\end{prop}
\begin{proof}
We write $\pi: C \times C \to \Sym^2(C)$ for the quotient map, and let $\pr_i : C^2 \to C$ be the projection to the $i$-th factor. For any point $P \in C$ we define a divisor on $C \times C$ 
\[
H_P  :=  \pr_1^{-1}(P)  + \pr_2^{-1}(P) 
\] 
which is normal crossing. Moreover, $H_P$ is ample and fixed by the involution switching the factors. Let $D_P$ be the quotient of $H_P$ by $S_2$, i.e.\ the image of $H_P$ in $\Sym^2(C)$. The divisor  $D_P$ is an ample effective divisor on $X$ because $\pi$ is finite and $H_P = \pi^\ast(D_P)$. For distinct points $P \not= Q$ in $C$, the divisor $D = D_P + D_Q$ is normal crossing. We compute
\[
(D_P \bullet  D_Q) = \frac{1}{\deg(\pi)} (H_P \bullet  H_Q)  = 1.
\]
The result now follows at once from  \pref{prop:NCDampleFujitaExtremeSurface}.
\end{proof}

\begin{rmk}
\begin{enumerate}[align=left,labelindent=0pt,leftmargin=*]
\item 
Note that $\Sym^2(C)$ for $C = \bP^1$ is just $\bP^2$. So we recover the convex Fujita number of this classical example. 
\item
For $C = E$ an elliptic curve, the summation map $\Sym^2(E) \to \Pic^2_E \simeq E$ is a $\bP^1$-bundle associated to the vector bundle $\cE = f_\ast \cL$ where $f: E \times \Pic^2_E \to \Pic_E^2$ is the projection and $\cL$ is the universal line bundle of degree $2$. By \pref{prop:FNruled surface} the vector bundle $\cE$ must be stable of odd degree.
\item
If $C$ is of genus at least $2$ and not hyperelliptic, then the summation map identifies $\Sym^2(C)$ with a closed subscheme of the Jacobian $\Pic^2_C \simeq \Pic^0_C$. The sequence of maps
\[
C \times C \to \Sym^2(C) \to \Pic^0_C
\]
shows that the convex Fujita number can go up and down along finite (ramified) maps: we have $\conFN(C \times C) = 2$, while $\conFN(\Sym^2(C)) = 3$, but again $\conFN(\Pic^0_C) = 2$. The latter holds because $\Pic^0_C$ is principally polarized, and an ample line bundle inducing the principal polarization is not globally generated. 
See \rref{rmk:FN in etale cover} for a related remark.
\end{enumerate}
\end{rmk}

\subsection{Some rational non-minimal Fujita extreme surfaces}

Here we present a construction that yields  rational Fujita extreme surfaces. This serves as a model for the construction in \secref{sec:Fujitaextremecofinal}.

A pencil of plane curves of degree $d$ arises by choosing a line in $\bP\big(\rH^0(\bP^2,\dO(d))\big)$. We would like our pencil to have only irreducible (and reduced) members. This translates to the following: the chosen  line should avoid the products obtained by multiplication of equations of degree $a,b \geq 1$ with $a+b= d$, i.e., the images of the finite map that multiplies equations
\[
\bP\big(\rH^0(\bP^2,\dO(a))\big) \times \bP\big(\rH^0(\bP^2,\dO(b))\big) \to \bP\big(\rH^0(\bP^2,\dO(d))\big).
\]
This is achieved by any general line by the following dimension count. First recall
\[
\dim \bP\big(\rH^0(\bP^2,\dO(d))\big) = \binom{d+2}{2} - 1 = \frac{d(d+3)}{2}\ .
\]
Then, using that linear terms cancel, we find
\begin{align*}
& \dim \bP\big(\rH^0(\bP^2,\dO(d))\big)  \geq 2 + \dim \bP\big(\rH^0(\bP^2,\dO(a))\big) + \dim \bP\big(\rH^0(\bP^2,\dO(b))\big)  \\
\iff & \frac{d(d+3)}{2} \geq 2 + \frac{a(a+3)}{2} + \frac{b(b+3)}{2} \iff  d^2 \geq 4 + a^2 + b^2 \iff ab \geq 2.
\end{align*}
So we succeed if $d \geq 3$ which we assume from now on.

With two general planar degree $d$ curves $D_0$ and $D_\infty$ in $\bP^2$, we denote by $Z = D_0 \cap D_\infty \setminus \{P\}$ the set of $d^2$ intersection points with a distinguished point $P$ removed. We consider the blow up $X = \Bl_Z(\bP^2)$ and the further blow up $\sigma : Y = \Bl_P(X) \to X$. After blowing up all of $D_0 \cap D_\infty$, a transversal intersection by assumption, we arrive at the pencil $f: Y \to \bP^1$ whose fibers are all irreducible and reduced plane curves of degree $d$. 

\begin{prop}
\label{prop:general pencil degree d}
The surface $X$ obtained by blowing up $\bP^2$ in $d^2-1$ points of intersection of general plane curves $D_0$ and $D_\infty$ of degree $d$ has convex Fujita number $\conFN(X) = 3$.
\end{prop}
\begin{proof}
We use the notation of the construction above. In particular we assume that all curves in the pencil defined by $D_0$ and $D_\infty$ are irreducible. 

Let $C_0$ and $C_\infty$ be the strict transform of $D_0$ and $D_\infty$ in $X$. Then the $C_i$ are linearly equivalent, they intersect transversally in exactly one point $P$ and we are going to show that they are ample. Since $C_i^2 = D_i^2 - (d^2-1) = 1$, we complete the proof in view of  the criterion of 
\pref{prop:NCDampleFujitaExtremeSurface}.

Let $\cL$ be the line bundle associated on $X$ to the $C_i$, and recall the blow up $\sigma : Y \to X$ in $P$ which realizes our pencil as a regular fibration $f: Y \to \bP^1$. Let $E = \sigma^{-1}(P)$ be the exceptional fiber, which induces a section of $f$. We already computed $\cL^2 = 1$, so to apply the Nakai-Moishezon criterion for ampleness it remains to compute the intersection number with an arbitrary irreducible curve $C$ in $X$.  Let $C'$ be the strict transform of $C$ in $Y$, and let $m$ be the multiplicity of $C$ in $P$. We have $\sigma^\ast C = C' + mE$ and $C' \bullet E = m$. Using $\sigma^\ast \cL (-E) \simeq f^\ast \dO(1)$ we have
\[
\cL \bullet C = \sigma^\ast \cL \bullet (C' + mE) = \sigma^\ast \cL \bullet C' = (f^\ast \dO(1) + E) \bullet C' = F \bullet C' + m
\]
where $F$ is a fiber of $f$. Since $F$ and $C'$ are irreducible and $F$ isotropic, we have $F \bullet C' \geq 0$, and $m \geq 0$ anyway. But indeed more is true: if $P \in C$, then $m > 0$; while if $P \notin C$, then $C'$ is not contained in a fiber (here we need that the fibers of $f$ are irreducible) and therefore $F \bullet C' > 0$. So the proof is complete.
\end{proof}

\begin{cor}
\label{cor:FNdelPezzo1}
A general del Pezzo surface $X$ of degree $1$ has convex Fujita number $3$.
\end{cor}
\begin{proof}
A general del Pezzo surface of degree $1$ arises by the construction above for $d=3$.
\end{proof}

\subsection{Fujita extreme surfaces are birationally cofinal} 
\label{sec:Fujitaextremecofinal}

\begin{theorem}
\label{prop:cofinalFN3}
Let $X$ be a smooth projective surface. Then there is a proper birational modification $X' \to X$ such that $X'$ has convex Fujita number $\conFN(X') = 3$.
\end{theorem}
\begin{proof}
As a first step we argue that there is a very ample divisor $L$ on $X$, giving rise to an embedding $X \inj \bP^N$, such that  any hyperplane section $D  = H \cap X$ in $\bP^N$ is $2$-connected. Indeed, we may choose an ample divisor $A$ on $X$ such that $L = 3A$ is very ample. By \cite[Theorem 1]{van_de_ven_2-connectedness_1979}  the hyperplane sections $D$ are $2$-connected, because the two exceptional cases of \textit{loc.\ cit.} lead to a divisor $C$ with intersection number $D \cdot C \leq 2$ while all intersection numbers with $L$ are divisible by $3$ (the exceptional cases are (1) two lines in $\bP^2$ --- here $C$ is a line, and (2) one section plus a sum of fibers --- here $C$ is another fiber).

Given the very ample $L$ as above, we study  the locus $B \subseteq \bP(\rH^0(X,L))$ of divisors $D$ in the complete linear series of $L$ that are reducible as divisors. In order to describe $B,$ we use the dual $\bP^N$, denoted by $\check{\bP}^N$, of hyperplanes $H \subseteq \bP^N$. The total space of tangent hyperplanes 
\[
TH := \{(H,P) \ ; \ \rT_P(X) = \rT_P(X \cap H)\} \subseteq \check{\bP}^N \times X
\]
is a closed subspace in $\check{\bP}^N \times X$. The image under the first projection $\pr: TH \to \check{\bP}^N$ is the dual variety $X^\ast$. The second projection $TH \to X$ is a Zariski locally trivial $\bP^{N-3}$-bundle showing that $TH$ is in fact connected, smooth and of dimension $N-1$. 

By the case $n=2$ and $m=N-1$ of \cite[Corollary 2.4]{zak_tangents_1993} the generic tangent hyperplane section to $X$ in $\bP^N$ is tangent to $X$ in exactly one point. This means that $\pr: TH \to X^\ast$ is generically finite with just one point in the fiber. In particular $TH \to X^\ast$ is in fact a birational map. Let $U \subseteq X^\ast$ be a dense open such that $\pr$ is an isomorphism restricted to $U$. Then for all hyperplanes $H \in U$ the fiber 
\[
\pr^{-1}(H) = \{ P \in X \ ; \ \rT_P(X) \subseteq H\}
\]
consists of a single reduced point. Therefore $D = X \cap H$ has only one singular point $P$, and if $D = D_1 + D_2$ is a sum of components, then $D_1$ and $D_2$ meet transversally in $P$. Therefore $D_1 \cdot D_2 = 1$ contradicting the fact that $L$ was chosen so that $D$ is $2$-connected. It follows that $B$, the locus of reducible $D \in |L|$  is contained in $X^\ast \setminus U$. So we can estimate the dimension of $B$ as
\[
\dim(B) < \dim X^\ast = \dim TH = N-1.
\]
Therefore a general line in $\bP(\rH^0(X,L))$ avoids $B$. Hence a general pencil of divisors in $|L|$ consists entirely of irreducible divisors. 

Let $D_0$ and $D_\infty$ be general divisors in $|L|$ such that the pencil spanned by these two divisors consists only of irreducible curves. Let $Z$ be the complement of a point $P$ of the intersection $D_0 \cap D_{\infty}$, and let $\sigma: X' \to X$ be the blow-up in $Z$. Then the strict transforms $D'_i$ of $D_i$ for $i \in \{0,\infty\}$ in $X'$ are ample with $D'_0 \cdot D'_\infty = 1$, essentially due to the Nakai-Moishezon criterion as in the proof of \pref{prop:general pencil degree d}. 
The criterion of \pref{prop:NCDampleFujitaExtremeSurface} completes the proof of $\conFN(X') = 3$.
\end{proof}

\begin{thm}
\label{thm:allpi1FN3}
For every projective group $\pi$ there is a smooth projective surface $X$ with $\pi_1(X) \simeq \pi$ and convex Fujita number $\conFN(X) =3$.
\end{thm}
\begin{proof}
It is well known that any projective group $\pi$ can be realized as the fundamental group of a smooth projective surface. We then dominate this surface birationally by one with convex Fujita number $3$ by \pref{prop:cofinalFN3}, noting that the fundamental group is a birational invariant and thus does not change.
\end{proof}

\begin{rmk}
\begin{enumerate}
\item
The surfaces constructed in the proof of \tref{thm:allpi1FN3} are not minimal in general. To determine whether $\conFN(X) = 3$ is possible for a minimal smooth projective surface with given birational invariants is a completely different matter.
\item
\tref{thm:allpi1FN3} describes Fujita extreme surfaces  ($\conFN(X) = 3$) with given fundamental group. In \cite[Proposition 4.1]{chen_convex_2023} we showed that also Fujita simple surfaces ($\conFN(X) = 0$) and surfaces with convex Fujita number $\conFN(X) = 1$ with given fundamental group exist. 
\end{enumerate}
\end{rmk}

\bibliographystyle{alpha}

\end{document}